%% file: pap_arxiv.tex
\documentclass[11pt]{article}

\title{Best Arm Identification for Contaminated Bandits}

\usepackage{amsmath,amssymb,graphicx,url}
\usepackage[boxed]{algorithm2e}
\SetAlCapSkip{1em}

\usepackage{amsthm}
\newtheorem{theorem}{Theorem}
\newtheorem{corollary}[theorem]{Corollary}
\newtheorem{lemma}[theorem]{Lemma}

\newtheorem{definition}{Definition}
\newtheorem{remark}[theorem]{Remark}

\usepackage{natbib}
\usepackage[margin=1.3in]{geometry} 

\usepackage{hyperref}       
\usepackage{url}            
\usepackage{booktabs}       
\usepackage{nicefrac}       
\usepackage{microtype}      
\usepackage{framed}

\usepackage{mathtools,dsfont}
\usepackage{bm}
\usepackage{xcolor}
\usepackage{graphicx}
\usepackage{enumerate}
\usepackage{wrapfig}

\newcommand{\Real}{\mathbb{R}}
\newcommand{\E}{\mathbb{E}}
\newcommand{\Prob}{\mathbb{P}}
\newcommand{\Ber}{\text{Ber}}
\newcommand{\SBer}{\text{SBer}}

\newcommand{\Unif}{\text{Unif}}
\newcommand{\UnifI}{\text{Unif}([0,1])}

\newcommand{\calA}{\mathcal{A}}

\DeclareMathOperator{\calF}{\mathcal{F}}

\newcommand{\half}{\tfrac{1}{2}}

\newcommand{\eps}{\varepsilon}
\newcommand{\logdel}{\log\tfrac{1}{\delta}}
\newcommand{\logtdel}{\log\tfrac{2}{\delta}}

\newcommand{\QLF}{Q_{L,F}}
\newcommand{\QRF}{Q_{R,F}}

\newcommand{\med}{{m_1}}
\newcommand{\mLtilde}{{\tilde{m}_{L}}}
\newcommand{\mRtilde}{{\tilde{m}_{R}}}
\newcommand{\hatm}{\hat{m}_1}
\newcommand{\hatmed}{\hatm} 

\newcommand{\hatsigmed}{\hat{m}_2}
\newcommand{\sigmed}{m_2}
\newcommand{\sigmedsq}{m_2^2}
\newcommand{\sigmedmax}{\bar{m}_2}

\newcommand{\epsmax}{\bar{\eps}}
\newcommand{\tmax}{\bar{t}}

\newcommand{\Famquantile}{\mathcal{H}_{\tmax, R}}
\newcommand{\Famnosig}{\mathcal{F}_{\tmax, B}}
\newcommand{\Fam}{\mathcal{F}_{\tmax, B, \sigmedmax,\kappa}}

\newcommand{\Uncertaintygen}{U_{\eps,B,\sigmed}}
\newcommand{\UncertaintyF}{U_{\eps,B,\sigmed(F)}}

\newcommand{\Uncertaintygenprescient}{U_{\eps,B,\sigmed}^{\textsc{(malicious)}}}
\newcommand{\UncertaintyFprescient}{U_{\eps,B,\sigmed(F)}^{\textsc{(malicious)}}}

\newcommand{\plusminus}{\raisebox{.2ex}{$\scriptstyle\pm$}}
\newcommand{\DS}{\displaystyle}
\newcommand{\PP}{\mathbb P}

\newcommand{\BAI}{\textsc{BAI}}
\newcommand{\CBAI}{\textsc{CBAI}}
\newcommand{\PIBAI}{\textsc{PIBAI}}

\newcommand{\BAIspace}{\textsc{BAI}\;}
\newcommand{\CBAIspace}{\textsc{CBAI}\;}
\newcommand{\PIBAIspace}{\textsc{PIBAI}\;}

\newcommand{\hatI}{\hat{I}}

\input{def}

\author{
        Jason Altschuler\\
        Massachusetts Institute of Technology\\
        \texttt{jasonalt@mit.edu}
        \and
        Victor-Emmanuel Brunel\\
        Massachusetts Institute of Technology\\
        \texttt{vebrunel@mit.edu}\\
        \and
        Alan Malek \\
        Massachusetts Institute of Technology\\
        \texttt{amalek@mit.edu}
}
\begin{document}

\date{}
\maketitle

\begin{abstract}
	This paper studies active learning in the context of robust statistics. Specifically, we propose a variant of the Best Arm Identification problem for \emph{contaminated bandits}, where each arm pull has probability $\eps$ of generating a sample from an arbitrary contamination distribution instead of the true underlying distribution. The goal is to identify the best (or approximately best) true distribution with high probability, with a secondary goal of providing guarantees on the quality of this distribution. The primary challenge of the contaminated bandit setting is that the true distributions are only partially identifiable, even with infinite samples. To address this, we develop tight, non-asymptotic sample complexity bounds for high-probability estimation of the first two robust moments (median and median absolute deviation) from contaminated samples. These concentration inequalities are the main technical contributions of the paper and may be of independent interest. Using these results, we adapt several classical Best Arm Identification algorithms to the contaminated bandit setting and derive sample complexity upper bounds for our problem. Finally, we provide matching information-theoretic lower bounds on the sample complexity (up to a small logarithmic factor).
\end{abstract}

\newpage
\tableofcontents
\newpage

\input{1_intro.tex}
\input{2_setup.tex}

\input{3_alg_general.tex}

\input{4_median.tex}
\input{5_quality.tex}
\input{6_conclusion}

\paragraph*{Acknowledgements}
We thank Marco Avella Medina and Philippe Rigollet for helpful discussions. JA is supported by NSF Graduate Research Fellowship 1122374.

\newpage

\bibliography{bandits}

\appendix
\input{app_robestdetails.tex}
\input{app_estimation.tex}
\input{app_algdetails.tex}
\input{app_lb.tex}

\end{document}

%% file: def.tex
\DeclareMathOperator*{\argmax}{arg\,max}

\DeclareBoldMathCommand{\I}{I}
\DeclareBoldMathCommand{\e}{e}
\DeclareBoldMathCommand{\f}{f}
\DeclareBoldMathCommand{\F}{F}
\DeclareBoldMathCommand{\g}{g}
\DeclareBoldMathCommand{\a}{a}
\DeclareBoldMathCommand{\b}{b}
\DeclareBoldMathCommand{\c}{c}
\DeclareBoldMathCommand{\d}{d}
\DeclareBoldMathCommand{\m}{m}
\DeclareBoldMathCommand{\p}{p}
\DeclareBoldMathCommand{\q}{q}
\DeclareBoldMathCommand{\r}{r}
\DeclareBoldMathCommand{\R}{\mathbb{R}}
\DeclareBoldMathCommand{\V}{V}
\DeclareBoldMathCommand{\x}{x}
\DeclareBoldMathCommand{\t}{t}
\DeclareBoldMathCommand{\X}{X}
\DeclareBoldMathCommand{\Y}{Y}
\DeclareBoldMathCommand{\z}{z}
\DeclareBoldMathCommand{\Z}{Z}
\DeclareBoldMathCommand{\M}{M}
\DeclareBoldMathCommand{\n}{n}
\DeclareBoldMathCommand{\ssigma}{\sigma}
\DeclareBoldMathCommand{\eepsilon}{\epsilon}
\DeclareBoldMathCommand{\SSigma}{\Sigma}
\DeclareBoldMathCommand{\OOmega}{\Omega}
\DeclareBoldMathCommand{\y}{y}
\DeclareBoldMathCommand{\U}{U}
\DeclareBoldMathCommand{\w}{w}
\DeclareBoldMathCommand{\W}{W}
\DeclareBoldMathCommand{\L}{L}

\DeclareBoldMathCommand{\s}{s}
\DeclareBoldMathCommand{\A}{A}
\DeclareBoldMathCommand{\B}{B}
\DeclareBoldMathCommand{\C}{C}
\DeclareBoldMathCommand{\D}{D}
\DeclareBoldMathCommand{\G}{G}
\DeclareBoldMathCommand{\P}{P}
\DeclareBoldMathCommand{\Q}{Q}
\DeclareBoldMathCommand{\mmu}{\mu}
\DeclareBoldMathCommand{\ones}{1}
\DeclareBoldMathCommand{\zeros}{0}

\newcommand{\Reals}{\mathbb R}

\newcommand{\TODO}[1]{
\ifmmode
\text{\textcolor{red}{TODO: #1}}
\else
\textcolor{red}{TODO: #1}
\fi
}

%% file: 1_intro.tex

\section{Introduction}\label{sec:intro}
Consider Pat, an aspiring machine learning researcher deciding between working in a statistics, mathematics, or computer science department. Pat's sole criterion is salary, so Pat surveys current academics with the goal of finding the department with highest median income. However, some subset of the data will be inaccurate: some respondents obscure their salaries for privacy, some convert currency incorrectly, and some do not read the question and report yearly instead of monthly salary, etc. How should Pat target his or her surveys, in an adaptive (online) fashion, to find the highest paying department with high probability in the presence of contaminated data?

In this paper, we study the Best Arm Identification (\BAI) problem for multi-armed bandits where observed rewards are not completely trustworthy. The multi-armed bandit problem has received extensive study in the last three decades \citep{lai1985asymptotically,bubeck2012regret}. 
We study the \emph{fixed confidence}, or ($\alpha, \delta$)-PAC, \BAIspace problem, in which the learner must identify an $\alpha$-suboptimal arm with probability at least $1 - \delta$, using as few samples as possible. Most \BAIspace algorithms for the fixed-confidence setting assume i.i.d. rewards from distributions with relatively strict control on the tails, such as boundedness or more generally sub-Gaussianity \citep{jamieson2014lil}. However, for Pat, the data are neither i.i.d. nor from a distribution with controllable tails. How can we model this data, and how can we optimally explore these arms?

To answer the first question, we turn to robust statistics, which has studied such questions for over fifty years. In a seminal paper, \cite{huber1964robust} introduced the contamination model, which we adapt to the multi-armed bandit model by proposing the \emph{Contaminated Best Arm Identification problem} (\CBAI). This is formally defined in Section~\ref{sec:cbai}, but is informally described as follows. There are $k \geq 2$ arms, each endowed with a fixed base distribution $F_i$ and arbitrary contamination distributions $G_{i,t}$ for $t\geq 1$. We place absolutely no assumptions on $G_{i,t}$. 
When arm $i$ is pulled in round $t$, the learner receives a sample that with probability $1 - \eps$ is drawn from the arm's true distribution $F_i$, and with the remaining probability $\eps$ is drawn from an arbitrary contamination distribution $G_{i,t}$. The goal is to identify the arm whose true distribution has the highest (or approximately highest) median. The median is the goal rather than the mean since the distributions are not assumed to have finite first moments; and even if they do, the true distributions' means are impossible to estimate from contaminated samples since the contaminations may be arbitrary. A key point is that suboptimality of the arms is based on the quality of the underlying true distributions $F_i$, not the contaminated distributions $\tilde{F}_{i,t}$ of the observed samples. Note that existing \BAIspace algorithms, fed with samples from $\tilde{F}_{i,t}$, will not necessarily work. 

This contamination model nicely fits Pat's problem: samples are usually trustworthy, but sometimes they are completely off and cannot be modeled by a distribution with controlled tails. Additionally, the nature of contamination changes with the respondent, and hence $G_{i,t}$ should be considered as time varying, which completely breaks the usual i.i.d.\ data assumption. Finally, Pat wants to determine the department with highest \emph{true} median salary, not highest \emph{contaminated} median (or mean) salary.

The contaminated bandit setup also naturally models many other situations that the classical bandit setup cannot, such as any bandit problem where data may be subject to measurement or recording error with some probability. For example, consider the canonical problem of optimal experiment design where we are measuring drug responses, but in a setting where samples can be corrupted or where test results can be incorrectly measured or recorded. Another example is testing new software features, where yet-unfixed bugs may distort responses but will be fixed before release. More scenarios are discussed in Subsection~\ref{subsec:previous-work} when comparing to other models in previous works, and in Subsection~\ref{subsec:cbai:adversary} after the formal definition of the \CBAIspace problem.

Importantly, the \CBAIspace problem is nontrivially harder than the \BAIspace problem since there are no consistent estimators for statistics (including the mean or median) of $F_i$ if the contaminations are allowed to be arbitrary. Under some mild technical assumptions on $F_i$, the contamination can cause the median of $\tilde F_{i,t}$ to be anywhere in an $\Theta(\eps)$-neighborhood of the median of $F_i$, and hence we can only determine the median of $F_i$ up to some unavoidable estimation bias, $U_i$, of order $\eps$ (see Section~\ref{sec:cbai} for details). This leads us to generalize our study to the more abstract \emph{Partially Identifiable Best Arm Identification} (\PIBAI) problem (defined formally in Section~\ref{sec:pibai}), which includes \CBAIspace as a special case.

This \PIBAIspace problem can be seen as an active-learning version of the classical problem of estimation under partial identifiability, which has been studied for most of the last century in the econometrics literature \citep{marschak1944random, manski2009identification}. A canonical problem in this field is trying to learn the age distribution of a population by only asking in which decade each subject was born; clearly the median age can only be learned up to certain unidentifiability regions.

\subsection{Our Contributions}
To the best of our knowledge, this is the first paper to consider the Best Arm Identification problem with arbitrary contaminations. 
This contaminated bandit setup models many practical problems that the classical bandit setup cannot.
However, developing algorithms for this setup requires overcoming the challenge of partial identifiability of the arms' true distributions. Indeed, it is not hard to show that the adversary's ability to inject arbitrary contaminations can render ``similar'' underlying distributions indistinguishable, even with access to infinite samples.
\par We analyze this \CBAIspace problem under three models of the adversary's power: the \emph{oblivious} adversary chooses all contamination distributions a priori; the \emph{prescient} adversary may choose the contamination distributions as a function of all realizations (past and future) of the true rewards and knowledge of when the learner will observe contaminated samples; and the \emph{malicious} adversary may, in addition, correlate when the learner observes contaminated samples with the true rewards. Subsection~\ref{subsec:cbai:adversary} gives formal definitions of these adversarial settings, as well as motivating examples and applications for each.
\par Our technical contributions can be divided into three parts:
\begin{enumerate}
\item[(i)] we prove tight, non-asymptotic sample complexity bounds for estimation of the first two robust moments (median and median absolute deviation) from contaminated samples,
\item[(ii)] we use the statistical results in (i) to develop efficient algorithms for the \CBAIspace problem and provide sample complexity upper bounds for the fixed-confidence setting, and
\item[(iii)] we prove matching information-theoretic lower bounds showing that our algorithms have optimal sample complexity (up to small logarithmic factors).
\end{enumerate}
We elaborate on each of these below.
\par \textbf{Contribution (i).} These concentration inequalities are the main technical contributions of the paper and may be of independent interest to the robust statistics community. We consider estimating statistics of a single arm from contaminated samples and show that although estimation of standard moments (mean, variance) is impossible, estimation of robust moments (median, median absolute deviation) is possible. Specifically, for each of the three adversarial models, we show that with probability at least $1 - \delta$, the empirical median of the contaminated samples lies in a region around the true median of width $U_i+E_{n,\delta}$, where $U_i$ is some unavoidable bias that depends on quantiles of $F_i$ and the power of the adversary, $n$ is the number of samples, and $E_{n,\delta}$ is a confidence-interval term that decreases at the optimal $\sqrt{\tfrac{\log 1/\delta}{n}}$ rate. Our results neatly capture the effect of the adversary's power by deriving different $U_i$ for each scenario, thereby precisely quantifying the hardness of the three different adversarial settings. We also present non-asymptotic sample complexity guarantees for estimation of the second robust moment, often called the Median Absolute Deviation (MAD), under all three adversarial settings. The MAD is a robust measure of the spread of $F_i$ and controls the width $U_i$ of the median's unidentifiability region.
\par \textbf{Contribution (ii).} We show that, surprisingly, several classical \BAIspace algorithms are readily adaptable to the \PIBAIspace problem. This suggests a certain inherent robustness of these classical bandit algorithms. We first present these algorithms for an abstract version of the \PIBAIspace problem so that our algorithmic results may be easily transferrable to other application domains with different statistics of interest or different contamination models. We then combine these general results with our statistical results from (i) to obtain PAC algorithms for \CBAI, the problem this paper focuses on. We give fixed-confidence sample complexity guarantees that mirror the sample complexity guarantees for \BAIspace in the classical stochastic multi-armed bandit setup. The main difference is that \BAIspace sample complexities depend on the suboptimality ``gaps'' $\Delta_i := p_{i^*} - p_i$ between the statistics of the optimal arm $i^*$ and each suboptimal arm $i$, whereas our \CBAIspace sample complexities depend on the suboptimality ``\emph{effective gaps}'' $\tilde{\Delta}_i := (p_{i^*} - U_{i^*}) - (p_{i} + U_i) = \Delta_i - (U_{i^*} + U_i)$, which account for the unavoidable estimation uncertainties in the most pessimistic way. We also show how to apply the MAD estimations results from (i) to obtain guarantees on the quality of the underlying distribution of the selected arm.
\par \textbf{Contribution (iii).} We prove matching information-theoretic lower bounds (up to a small logarithmic factor) on the sample complexity of \CBAIspace via a reduction to classical lower bounds for the stochastic multi-armed bandit problem. We argue that for \CBAIspace the effective gap $\tilde\Delta_i$ is the right analog of the traditional gap since it appears in matching ways in both our upper and lower bounds.

\subsection{Previous Work}\label{subsec:previous-work}
The Best Arm Identification problem in the fixed-confidence setting has a long history, starting from work by \cite{bechhofer1968sequential} and \cite{lai1985asymptotically}. Recent interest from the learning theory community was sparked by the seminal paper by \cite{even2002pac}, which proposed several algorithms obtaining instance-adaptive sample complexity bounds. Since then, there has been significant work on the algorithmic side
\citep[see work by][]{kalyanakrishnan2012pac,gabillon2012best,karnin2013almost,jamieson2013finding,jamieson2014lil}. Concurrently, a parallel line of work has focused on improving lower bounds, starting with the $2$-armed setting~\citep{chernoff1972sequential, anthony2009neural}, extending to the multi-armed setting~\citep{mannor2004sample}, and, more recently, continuing with more finely tuned lower bounds that include properties of the arm distributions aside from the gaps~\citep{chen2015optimal,kaufmann2016complexity,garivier2016optimal}.

In the cumulative regret setting, the online learning literature has considered both stochastic bandits with mild tail assumptions (for example, \cite{bubeck2013bandits} only assumed the existence of a $(1+\eps)$ moment) and algorithms that obtain near-optimal regret guarantees if the environment is stochastic or adversarial \citep{bubeck2012regret}. The partial monitoring problem, where the learner only knows the loss up to some subset~\citep{bartok2014partial}, is also loosely similar to the \CBAIspace problem in the sense that both problems feature partial identification. We note, however, that minimizing cumulative regret is not a reasonable goal in our contamination setup since the contaminations can be arbitrary and even unbounded. (Furthermore, even when boundedness is assumed, minimizing cumulative regret can still be a poor criteria; see below for concrete examples.)

The existing literature closest to our work studies the \BAIspace problem in settings more general than i.i.d.\ arms, for example stochastic but non-stationary distributions \citep{Allesiardo17NS,Allesiardo2017SL} or arbitrary rewards where each arm converges to a limit \citep{jamieson2016non,li2016hyperband}. However, neither setting fits the contamination model or allows for arbitrary perturbations.
\par A related contamination model is studied by~\cite{seldin2014one}; however, they make a boundedness assumption which makes their setup drastically different from ours.\footnote{Another difference is that our setup and results also work for more powerful types of contaminating adversaries.
} Specifically, they consider cumulative regret minimization in a setting where rewards are bounded in $[0,1]$ and contaminations can be arbitrary. Because of this $[0,1]$ boundedness, contamination can move the means by at most $\eps$, which is why it is reasonable that they base their algorithms on the contaminated means. However, the best contaminated mean is \emph{not} a reasonable proxy for the best mean without the $[0,1]$ boundedness assumption. First off, the contaminated distributions may not even have finite first moments. Moreover, when rewards are bounded but only within a large range, the contaminated mean can deviate from the true mean by $\eps$ times the size of that range. This can lead to an error bound that is extremely loose -- sometimes to the point of being useless for prediction or estimation -- compared to the tight bound given by quantiles (see Lemma~\ref{lem:change-med}) which our algorithms achieve (see Subsection~\ref{subsec:cbai-pac}).
\par For instance, consider the adaptive survey example mentioned earlier. There might be a wide range of (true) academic salaries, say between zero and a million dollars. As such, contaminating an $\eps$ fraction of the data could move the mean by roughly $10^6 \cdot \eps$; even for a reasonable value of $\eps = 0.1$, this error bound of roughly $10^5$ is so large that it may drown out the actual information that the survey was trying to investigate. On the other hand, the median may be moved a significantly smaller amount since the (true) distribution of academic salaries might have reasonably narrow quantiles around the median (for example, being somewhat bell-shaped is sufficient). This performance improvement is intuitively explained by the fact that most academic salaries are not at the extremes of $0$ or a million dollars, but rather are fairly regular around the median. An identical phenomenon occurs in the canonical problem of optimal experiment design where we are measuring drug responses but samples can be corrupted or where test results can be incorrectly measured or recorded. Indeed, if the measurements of interests are, say, blood pressure or weight, the values might only be bounded within a large range, but the true distribution might have reasonably narrow quantiles around the median.
\par We note that the need for robust bandit algorithms is further motivated by the recent work of \cite{jun2018adversarial}, which shows that in a similar model, an adversary can make certain classical bandit algorithms perform very poorly by injecting only a small amount of contamination into the observed samples.

This paper also makes connections between several long bodies of work. The contamination model~\citep{huber1964robust} has a long history of more than fifty years in robust statistics; for examples, see work by \cite{hampel1974influence, maronna1976robust, rousseeuw2005robust, hampel2011robust}. Contamination models and malicious errors also have a long history in computer science, including the classical work of \cite{valiant1985learning,kearns1993learning} and a recent  burst of results on algorithms that handle estimation of means and variances \citep{lai2016agnostic}, efficient estimation in high dimensions \citep{diakonikolas2018robustly},
 PCA \citep{cherapanamjeri2017thresholding},
and general learning \citep{charikar2017learning} in the presence of outliers or corrupted data. Finally, the partial identification literature from econometrics also has a rich history \citep{marschak1944random, horowitz1995identification, manski2009identification,romano2010inference,bontemps2012set}.


\subsection{Notation}\label{subsec:notation}
We denote the Dirac measure at a point $x \in \Real$ by $\delta_x$, the Bernoulli distribution with parameter $p \in [0,1]$ by $\Ber(p)$, and the uniform distribution over an interval $[a,b]$ by $\Unif([a,b])$. The interval $[a-b, a+b]$ is denoted by $[a\plusminus b]$, the set of non-negative real numbers by $\Real_{\geq 0}$, the set of positive integers by $\mathbb{N}$, and the set $\{1, \dots, n\}$ by $[n]$ for $n \in \mathbb{N}$. We abbreviate ``with high probability'' by ``w.h.p.'' and ``cumulative distribution function'' by ``cdf''. 
\par Let $F$ be a cdf. We denote its left and right quantiles, respectively, by $\QLF(p) := \inf\{x \in \Real:F(x) \geq p \}$ and $\QRF(p) := \inf\{x \in \Real:F(x) > p \}$; the need for this technical distinction arises when $F$ is not strictly increasing. We denote the set of medians of $F$ by $\med(F) := [\QLF(\half), \QRF(\half)]$. When $F$ has a unique median, we overload $\med(F)$ to be this point rather than a singleton set containing it. For shorthand, we often write $m_1(X)$ for a random variable $X$ to denote $m_1(F)$, where $F$ is the law of $X$. When $F$ has a unique median, we denote the median absolute deviation (MAD) of $F$ by $\sigmed(F) := \med(|X - m_1(F)|)$ where $X \sim F$. When $m_2(F)$ is unique, we define $m_4(F) := \med(||X - m_1(F)| - m_2(F)|)$. Note that $m_1(F)$, $m_2(F)$, and $m_4(F)$ are robust analogues of centered first (mean), second (variance), and fourth (kurtosis) moments, respectively.
\par The empirical median of a (possibly random) sequence $x_1, \dots, x_n \in \Real$ is denoted by $\hatmed(x_1, \dots, x_n)$: if $n$ is odd, this is the middle value; and if $n$ is even, it is the average of the middle two values. The empirical MAD $\hatsigmed(x_1, \dots, x_n)$ is then defined as $\hatmed(|x_1 - \hatmed(x_1, \dots, x_n)|, \dots, |x_n - \hatmed(x_1, \dots, x_n)|)$.

\subsection{Outline}
Section~\ref{sec:cbai} formally defines the \CBAIspace problem and the power of the adversary, describes several motivating applications and examples, and discusses the primary challenge of the problem: partial identifiability. Section~\ref{sec:pibai} presents our algorithmic results for the abstract setting of best arm identification under partial identifiability (the \PIBAIspace problem). We state these algorithms in this general setting so that they may be easily transferrable later to other application domains with different statistics of interest or different contamination models. Section~\ref{sec:median} then specializes to the \CBAIspace problem. In order to implement the aforementioned general-purpose \PIBAIspace algorithms for \CBAI, concentration results are needed for estimation of medians given contaminated samples. These statistical results are developed in Subsection~\ref{subsec:contaminated.estimation}, and then used to derive upper bounds on the sample complexity of our \CBAIspace algorithms in Subsection~\ref{subsec:cbai-pac}. Subsection~\ref{subsec:lb} gives matching information-theoretic lower bounds on the sample complexity, showing that our algorithms are optimal (up to small logarithmic factors). This answers the primary question the paper sets out to solve: understanding the complexity of identifying the arm with best median given contaminated samples. Section~\ref{sec:quality} then turns to our secondary goal: providing guarantees on the distribution of the selected arm. This goal requires additionally estimating the second robust moment. Subsection~\ref{subsec:mad-finite} develops the necessary statistical results, which are then used in Subsection~\ref{subsec:cbai-quantile} provide our algorithmic results. Section~\ref{sec:conclusion} concludes and discusses several open problems. 

%% file: 2_setup.tex

\section{Formal Setup}
\label{sec:cbai}

Here we formally define the \textit{Contaminated Best Arm Identification problem (CBAI)}. Let $k \geq 2$ be the number of arms, $\eps \in (0, \half)$ be the contamination level, $\{F_i\}_{i \in [k]}$ be the true but unknown distributions, and $\{G_{i,t}\}_{i \in [k], t \in \mathbb{N}}$ be arbitrary contamination distributions. This induces contaminated distributions $\tilde{F}_{i,t}$, samples from which are equal in distribution to $(1 - D_{i,t})Y_{i,t} + D_{i,t}Z_{i,t}$, where $D_{i,t} \sim \Ber(\eps)$, $Y_{i,t} \sim F_i$, and $Z_{i,t} \sim G_{i,t}$. Note that if all of these random variables are independent, then each $\tilde{F}_{i,t}$ is simply equal to the contaminated mixture model $(1-\eps)F_{i} + \eps G_{i,t}$. However, we generalize by also considering the setting when the $Y_{i,t}, D_{i,t}, Z_{i,t}$ are not all independent, which allows an adversary to further obfuscate samples by adapting the distributions of the $D_{i,t}$ and $Z_{i,t}$ based on the realizations of the $Y_{i,t}$  (that is, by coupling these random variables); see below for details.

At each iteration $t$, a \CBAIspace algorithm chooses an arm $I_t \in [k]$ to pull and receives a sample $X_{I_t,t}$ distributed according to the corresponding contaminated distribution $\tilde{F}_{I_t,t}$. After $T$ iterations (a possibly random stopping time that the algorithm may choose), the algorithm outputs an arm $\hat{I} \in [k]$. For $\alpha \geq 0$ and $\delta \in (0,1)$, the algorithm is said to be $(\alpha,\delta)$\textit{-PAC} if, with probability at least $1 - \delta$, $\hatI$ has median within $\alpha + U$ of the optimal; that is,
\begin{align}
\Prob\left(m_1(F_{\hatI}) \geq \max_{i \in [k]} m_1(F_i) - (\alpha + U) \right)
\geq 1 - \delta,
\label{def:pac}
\end{align}
where $U$ is the unavoidable uncertainty term in median estimation that is induced by partial identifiability (see Subsection~\ref{subsec:cbai:partial-identifiability} for a discussion of $U$, and see Subsection~\ref{subsec:contaminated.estimation} for an explicit computation of this quantity). Thus, the goal is to find an algorithm achieving the PAC guarantee in \eqref{def:pac} with small expected sample complexity $T$.

\subsection{Power of the Adversary}\label{subsec:cbai:adversary} As is typical in online learning problems, it is important to define the power of the adversary since this affects the complexity of the resulting problem. Interestingly, \CBAIspace is still possible even when we grant the adversary significant power. We consider three settings, presented in increasing order of adversarial power. The key differences between these different types of adversaries are twofold: (1) whether they can choose the contaminated distributions ``presciently'' based on all other realizations $\{Y_{i,t}, D_{i,t}\}_{i \in [k], t \geq 1}$ both past and future; and (2) whether they can ``maliciously'' couple the distributions of each $D_{i,t}$ with the corresponding $Y_{i,t}$, subject only to the constraint that the marginals $D_{i,t} \sim \Ber(\eps)$ and $Y_{i,t} \sim F_i$ stay correct. 
\begin{itemize}
\item \textbf{Oblivious adversary.} For all $i \in [k]$, the triples $\{(Y_{i,t}, D_{i,t}, Z_{i,t})\}_{t \geq 1}$ are independent, and for all $t \geq 1$, $Y_{i,t}\sim F_i$, $D_{i,t}\sim \Ber(\varepsilon)$, and $Y_{i,t}$ and $D_{i,t}$ are independent. 
\item \textbf{Prescient adversary.} Same as the oblivious adversary except that the contaminations $Z_{i,t}$ may depend on everything else (i.e. $\{Y_{j,s},D_{j,s},Z_{j,s} \}_{j \in [k], s \geq 1}$). Formally, for all $i \in [k]$, the pairs $\{(Y_{i,t},D_{i,t})\}_{t\geq 1}$ are independent, and for all $t \geq 1$, $Y_{i,t}\sim F_i$, $D_{i,t}\sim \Ber(\varepsilon)$, $Y_{i,t}$ and $D_{i,t}$ are independent, and $Z_{i,t}$ may depend on all $\{Y_{j,s},D_{j,s},Z_{j,s}\}_{j\in [k], s \geq 1}$.

\item \textbf{Malicious adversary.} Same as the prescient adversary except that $Y_{i,t}$ and $D_{i,t}$ do not need to be independent. Formally, for all $i\in [k]$, the pairs $\{(Y_{i,t},D_{i,t})\}_{t\geq 1}$ are independent, and for all $t \geq 1$, $Y_{i,t}\sim F_i$, $D_{i,t}\sim \Ber(\varepsilon)$, and $Z_{i,t}$ may depend on all $\{Y_{j,s},D_{j,s},Z_{j,s}\}_{j\in [k], s \geq 1}$.

\end{itemize}
We note that for all three of these adversarial settings, we do not require independence between the outputs $(Y_{i,t}, D_{i,t}, Z_{i,t})$ across the arms for each fixed $t$.
\par The oblivious adversary is well-motivated, as many real-world problems where contaminations may occur fit naturally into this model. Indeed, this setting encompasses any bandit problem where data may be subject to measurement or recording error with some probability, such as measuring drug responses for clinical trials~\citep{lai1985asymptotically}, conducting surveys~\citep{national2017principles}, or testing new software features with yet-unfixed bugs as mentioned in the introduction. The theory we develop extends naturally and with little modification for the prescient and malicious adversaries. Moreover, these different settings allow for the modeling of many other real-world scenarios in which the contamination models are not as simple. For instance, all data may already exist but not yet be released to the learning algorithm. If adversaries have access to all the data from the beginning, they may be able to contaminate samples depending on all the data -- this is captured by the prescient setting. The malicious setting includes, for example, scenarios in which adversaries (for example, hackers) can successfully contaminate a given sample depending on the amount of effort they put into it, and the adversaries can try harder to contaminate certain samples depending on their values.
\par Perhaps surprisingly, we will show that the sample complexity is the same for both oblivious and prescient adversaries. Indeed for these two settings, we will prove our upper bounds for the more powerful setting of prescient adversaries, and our lower bounds for the less powerful setting of oblivious adversaries. We also note that the rate for malicious adversaries is only worse by at most a ``factor of $2$''; see Section~\ref{sec:median} for a precise statement. 

\subsection{The Challenge of Partial Identifiability of the Median}\label{subsec:cbai:partial-identifiability}
\label{sec:median-infinite}
In the introduction, we emphasized the point that a contaminating adversary can render different underlying distributions of an arm statistically indistinguishable. That is, even with infinite samples, it is impossible to estimate statistics of an arm's true distribution exactly. Consider, for example, the problem of estimating the median of a single arm with true distribution $F$ against an oblivious adversary, which is the weakest of our three adversarial settings. Define $S$ to be the set of all distributions $F'$ for which there exists adversarially chosen distributions $G$ and $G'$ such that $(1 - \eps)F + \eps G=(1 - \eps)F' + \eps G'$. How large is this set $S$, and in particular, how far can the medians of distributions in $S$ be from the median of $F$?

The following simple example shows that $S$ is non-trivial (and thus in particular contains more than just $F$). Let $F$ be the uniform distribution on the interval $[-1,1]$, and let $G$ be the uniform distribution on $[-1-c,-1]\cup[1,1+c]$, where $c=\varepsilon(1-\varepsilon)^{-1}$. The contaminated distribution $\tilde F:=(1-\varepsilon)F+\varepsilon G$ is the uniform distribution on $[-1-c,1+c]$. However, for any $p\in [-c,c]$, $\tilde F$ is also equal to $(1-\varepsilon)F(\cdot-p)+\varepsilon G_p$, where $G_p$ is the uniform distribution over $[-1-c,1+c]\setminus [-1+p,1+p]$. We conclude that $F$ is statistically indistinguishable from any of the translations $\{F(\cdot - p) \}_{p \in [-c,c]}$. Hence, $S$ is non-trivial and contains distributions with medians at least $\Omega(c)$ away from $\med(F)$. Even in this simple setting, infinite samples only allow us to identify the median of $F$ at most up to the unidentifiability region $[-c, c]$.

In fact, this toy example captures the correct dependence on $\varepsilon$ of how far the median of the contaminated distribution can be shifted, as formalized by the following simple lemma.
\begin{lemma}\label{lem:change-med}
For any $\eps \in (0, \half)$, any distribution $F$, and any median $m \in \med(F)$,
\[
\sup_{\substack{\mathrm{distribution}\; G, \\ \tilde{m} \in \med((1 - \eps) F + \eps G)}}
\left|\tilde{m} - m \right| = 
\max\left\{
\QRF\left(\frac{1}{2(1-\eps)} \right) - m, 
m - \QLF\left(\frac{1-2\eps}{2(1-\eps)}\right)
\right\}.
\]
\end{lemma}
\begin{proof}
  The fact that the left hand side is no smaller than the right hand side is straightforward: to shift the median to the right (resp. left), let $\{G_n\}_{n \in \mathbb{N}}$ be a sequence of distributions which are Dirac measures at $n$ (resp. $-n$).

  We now prove the reverse direction, the ``$\leq$'' inequality for any fixed distribution $G$. For shorthand, denote the contaminated distribution by $\tilde{F} := (1 - \eps) F + \eps G$, and denote its left and right medians by $\mLtilde := Q_{L,\tilde{F}}(\half)$ and $\mRtilde := Q_{R,\tilde{F}}(\half)$, respectively. Since every median of $\tilde{F}$ lies within $[\mLtilde, \mRtilde]$, it suffices to show that $\mLtilde \geq  \QLF(\tfrac{1-2\eps}{2(1-\eps)})$ and $\mRtilde \leq \QRF(\tfrac{1}{2(1-\eps)})$. We bound $\mLtilde$ presently, as bounding $\mRtilde$ follows by a similar argument or by simply applying the $\mLtilde$ bound to $F(-\cdot)$. By definition of $\mLtilde$, for all $t > 0$, $\half \leq \tilde{F}(\mLtilde + t) = (1 - \eps)F(\mLtilde + t) + \eps G(\mLtilde + t) \leq (1 - \eps)F(\mLtilde + t) + \eps$, where the last step is because $G(\mLtilde + t) \leq 1$ since $G$ is a distribution. Rearranging yields $F(\mLtilde + t) \geq  \tfrac{1-2\eps}{2(1 - \eps)}$, which implies that $\mLtilde \geq \QLF(\tfrac{1-2\eps}{2(1 - \eps)})$.
\end{proof}

\par Finally, we note that for some simple contamination models, partial identifiability in \CBAIspace is not a problem for identifying the arm with the best median. For example, if all arms are contaminated with a common distribution, then although the true medians are only partially identifiable, the contaminated medians are ordered in the same way as the true medians, albeit perhaps not strictly. However, for general (arbitrary) contaminations, such an ordering invariance is clearly not guaranteed.

%% file: 3_alg_general.tex

\section{Algorithms for Best Arm Identification Under Partial Identifiability}\label{sec:pibai}
Our algorithms for \CBAIspace actually work for the more general problem of best arm identification in \textit{partially identified} settings. Specifically, consider any setting where the statistic (for example, the median or mean) which measures the goodness of an arm can be estimated only up to some unavoidable error term due to lack of identifiability. The main result of this section is informally that certain \BAIspace algorithms for the classical stochastic multi-armed bandit setting can be adapted with little modification to such partially identified settings. Perhaps surprisingly, this suggests a certain innate robustness of these existing classical \BAIspace algorithms.

We present our algorithms in this section for this slightly more abstract problem of \textit{Partially Identifiable Best-Arm-Identification} problem (\PIBAI), which we define formally below. This abstraction allows our algorithmic results to later be transferred easily to other variants of the Best Arm Identification problem with different contamination models or different statistics of interest. We will show in Section~\ref{sec:median} how to adapt this to the \CBAIspace problem, the main focus of the paper, after first developing the necessary statistical results.
\par We now formally define the setup of \PIBAI. Let $k \geq 2$ be the number of arms. For each arm $i\in [k]$, consider a family of distributions $\mathcal D_i=\{D_i(p_i,G)\}_{G\in\mathcal G}$ where $p_i$ is a real-valued measure of the quality of arm $i$ and $G$ is a nuisance parameter in some abstract space $\mathcal G$. We let $i^* := \argmax_{i \in [k]} p_i$ be the best arm. We assume the existence of non-negative unavoidable biases $\{U_i\}_{i \in [k]}$ satisfying 
\begin{itemize}
\item[(i)] even from infinitely many independent samples $X_t, t=1,2, \ldots$ with $X_t\sim D_i(p_i,G_t)$ for some unknown, possibly varying $G_t\in\mathcal G$ ($t\geq 1$), it is impossible to estimate $p_i$ more precisely than the region $[p_i\, \plusminus\, U_i]$, and
\item[(ii)] there exists some estimator that, for any $\alpha > 0$ and $\delta \in (0,1)$, uses $n_{\alpha, \delta} = O(\alpha^{-2}\logdel)$ i.i.d. samples\footnote{Using the same techniques as presented in this paper, one can also consider \PIBAIspace for general $n_{\alpha,\delta} \neq O(\alpha^{-2}\logdel)$ and compute the resulting sample complexities. However, for simplicity of presentation, we assume that $n_{\alpha,\delta} = O(\alpha^{-2}\logdel)$ since anyways this is the natural (and optimal) quantity for many estimation problems such as estimating a median from contaminated samples (see Corollaries~\ref{corol:est-med} and~\ref{corol:est-med-malicious}), or using Chernoff bounds to estimate the mean of a $[0,1]$-supported distribution for classical stochastic MAB, etc.} from $D_i$ to output an estimate $\hat{p}_i$ satisfying
\begin{align}
\Prob\Big(
\hat{p}_i \in [p_i \,\plusminus\, (U_i + \alpha)]
\Big)
\geq 1 - \delta.
\label{eq:estimator}
\end{align}
\end{itemize}
The \PIBAIspace problem is then precisely the standard fixed-confidence stochastic bandit problem using these distributions where in each iteration $t$, the algorithm chooses an arm $I_t \in [k]$ and receives a sample from $D_{I_t}(p_{I_t},G_t)$ for some unknown $G_t\in\mathcal G$.

By the partial identifiability property (i), it is clear that even given infinite samples, it is impossible to distinguish between the optimal arm $i^*$ and any suboptimal arm $i \neq i^*$ satisfying $p_i + U_i \geq p_{i^*} - U_{i^*}$. Therefore, we assume henceforth the statistically possible setting in which the \emph{effective gaps} $\tilde\Delta_i := (p_{i^*} - U_{i^*}) - (p_i + U_i)$ are strictly positive for each suboptimal arm $i \neq i^*$.

For any $\alpha \geq 0$, arm $i$ is said to be $\alpha$\textit{-suboptimal} if $\tilde\Delta_i \leq \alpha$. Moreover, for any $\alpha \geq 0$ and $\delta \in (0,1)$, a \PIBAIspace algorithm is said to be $(\alpha, \delta)$\textit{-PAC} if it outputs an arm $\hat{I}$ that is $\alpha$-suboptimal with probability at least $1 - \delta$. That is,
\[
\Prob\big(
\tilde\Delta_{\hat{I}} \leq \alpha
\big)
\geq 1 - \delta,
\]
where the above probability is taken over the possible randomness of the samples, the estimator from (ii), and the \PIBAIspace algorithm. 
\par We now present algorithms for the \PIBAIspace problem. First, in Subsection~\ref{subsec:alg-simple}, we present a simple algorithm that performs uniform exploration among all arms. Next, in Subsection~\ref{subsec:alg-instance-adaptive}, we obtain more refined, instance-adaptive sample complexity bounds by adapting the \textsc{Successive Elimination} algorithm of~\citep{even2006action}. All algorithms use as a blackbox an estimator satisfying the above property (ii) of the \PIBAIspace problem.

\subsection{Simple Algorithm}\label{subsec:alg-simple}
\begin{algorithm}[t]
\For{$i \in [k]$}{
Sample arm $i$ for $n_{\alpha/2,\delta/k}$ times and produce estimate $\hat{p}_i$
}
Output $\hatI := \max_{i \in [k]} \hat{p}_i$.
\caption{Simple uniform exploration algorithm for $\PIBAI$.}
\label{alg:Simple}
\end{algorithm}
A simple ($\alpha, \delta$)-PAC \PIBAIspace algorithm is the following: pull each of the $k$ arms $n_{\alpha/2, \delta/k}$ times to create estimates $\hat{p}_i$ and output the arm $\hat{I} := \max_{i \in [k]} \hat{p}_i$ with the highest estimate. Pseudocode is given in Algorithm~\ref{alg:Simple}.

\begin{theorem} \label{Thm:Simple}
For any $\alpha>0$ and $\delta\in (0,1)$,  Algorithm~\ref{alg:Simple} is an $(\alpha,\delta)$-PAC \PIBAIspace algorithm with sample complexity $O(k n_{\alpha/2,\delta/k}) = O\left(\frac{k}{\alpha^2}\log \tfrac{k}{\delta} \right)$.
\end{theorem}
\begin{proof}
By \eqref{eq:estimator} and a union bound, we have that with probability at least $1 - \delta$, all estimates $\hat{p}_i \in [p_i \plusminus (U_i + \tfrac{\alpha}{2})]$. Whenever this occurs,
\[
\tilde\Delta_{\hat{I}}
= 
(p_{i^*} - U_{i^*}) - (p_{\hat{I}} + U_{\hat{I}})
\leq
(\hat{p}_{i^*} + \tfrac{\alpha}{2}) - (\hat{p}_{\hat{I}} - \tfrac{\alpha}{2})
\leq \alpha,
\]
implying that $\hat{I}$ is $\alpha$-suboptimal. The sample complexity bound is clear.
\end{proof}

\subsection{Instance-adaptive Algorithms}\label{subsec:alg-instance-adaptive}
The sample complexity of the simple Algorithm~\ref{alg:Simple} is not adaptive to the difficulty of the actual instance: an arm is sampled $n_{\alpha/2, \delta/k}$ times even if it is far from $\alpha$-suboptimal (meaning that $\tilde{\Delta}_i \gg \alpha$) and could potentially be eliminated much more quickly. In this subsection, we obtain such an instance-adaptive sample complexity by modifying the Successive Elimination algorithm of \citep{even2006action}, which was originally designed for the classical multi-armed bandit problem.
\par Pseudocode is given in Algorithm~\ref{alg:SuccElim}. At each round $r$, Algorithm~\ref{alg:SuccElim} gets a single new sample from each remaining arm in order to update its estimate $\hat p_{i,r}$. Then, for the next round $r+1$, it only keeps the arms $i$ whose estimates $\hat p_{i,r}$ are $\alpha$-close to the best estimate, where the threshold $\alpha$ is updated at each round.
\begin{algorithm}[t]
$S \leftarrow [k]$, $r \leftarrow 1$
\\ \While{$|S| > 1$}{
Sample each arm $i \in S$ once and produce $\hat{p}_{i,r}$ from all $r$ past samples of it \\
$S \leftarrow \{i\in S\,:\, \hat p_{i,r} \geq \max_{j\in S}\hat p_{j,r} - 2\alpha_{r, 6\delta/(\pi^2 k r^2)} \}$\\
$r \leftarrow r+1$
}
Output the only arm left in $S$
\caption{
  Adaptation of Successive Elimination algorithm for $\PIBAI$. Here, $\alpha_{r,\delta} := \sqrt{\tfrac{c\logdel}{r}}$, where $c$ is a universal constant satisfying $n_{\alpha,\delta} \leq c\alpha^{-2}\logdel$ (see Equation~\ref{eq:estimator}).
}
\label{alg:SuccElim}
\end{algorithm}
The algorithm and analysis are almost identical to the original. As such, proof details are deferred to Appendix~\ref{app:algorithms}. The main difference is that in the proof of correctness, we show the event $\{|\hat p_{i,r}-p_i|\leq U_i+ \alpha_{r,6\delta/(\pi^2 k r^2)},\;  \forall r \in [R] ,\;\forall i\in S_r \}$ occurs with probability at least $1 - \delta$. This ensures that in each round $r$, each estimate $\hat{p}_{i,r}$ is accurate enough to use for the elimination step.

Note that Algorithm~\ref{alg:SuccElim} returns the optimal arm w.h.p. (without knowing the smallest effective gap), unlike Algorithm~\ref{alg:Simple} above which only returns a near-optimal arm w.h.p.

\begin{theorem}\label{Thm:SuccElim}
Let $\delta\in (0,1)$. With probability at least $1 - \delta$, Algorithm~\ref{alg:SuccElim} outputs the optimal arm after using at most $O\left(\sum_{i\neq i^*} \tfrac{1}{\tilde\Delta_i^2}\log\left(\frac{k}{\delta\tilde\Delta_i}\right)\right)$ samples.
\end{theorem}

Moreover, as noted in Remark 9 of \citep{even2006action}, Algorithm \ref{alg:SuccElim} is easily modified (by simply terminating early) to be an ($\alpha, \delta$)-PAC algorithm with sample complexity \[
O\left(\frac{N_{\alpha}}{\alpha^2} \log\left(\frac{N_\alpha}{
\delta}\right) + \sum_{i\in [k]\,:\, \tilde\Delta_i>\alpha} \frac{1}{\tilde\Delta_i^2} \log\left(\frac{k}{\delta\tilde\Delta_i}\right)\right),
\]
where $N_\alpha$ is the number of $\alpha$-suboptimal arms with $\tilde\Delta_i \leq \alpha$.



\begin{remark}\label{rem:adapt}
In the multi-armed bandit literature, the sample complexity of the Successive Elimination algorithm (tight up to a logarithmic factor) was improved upon by~\citep{karnin2013almost}'s Exponential-Gap Elimination (EGE) algorithm (tight up to a doubly logarithmic factor). A natural idea is to analogously improve upon the \PIBAIspace guarantee in Theorem~\ref{Thm:SuccElim} by adapting the EGE algorithm. However, this approach does not work. The same holds for adapting the Median Elimination algorithm of~\citep{even2006action} and the PRISM algorithm of~\citep{jamieson2013finding}. The reason for the inadaptability of these algorithms -- in contrast to the easy adaptability of the Successive Elimination algorithm above -- is that these algorithms heavily rely on the ``additive property of suboptimality'' for \BAI:
\begin{itemize}
\item[] ``If arm $i$ has $\Delta_i$ suboptimality gap w.r.t. the optimal arm $i^*$, and if arm $j$ has $\Delta_j^{(i)}$ suboptimality gap w.r.t. arm $i$, then arm $j$ has suboptimality gap $\Delta_j = \Delta_i + \Delta_j^{(i)}$ w.r.t. the optimal arm $i^*$.''
\end{itemize}
The critical point is that \PIBAIspace does not have this property since errors propagate from adding the uncertainties $U_i$ in the suboptimality gaps: that is, $\DS \left(p_{i^*}-U_{i^*}\right)-\left(p_{j}+U_j\right)\neq \left[\left(p_{i^*}-U_{i^*}\right)-\left(p_{i}+U_i\right)\right]+\left[\left(p_{i}-U_{i}\right)-\left(p_{j}+U_j\right)\right]$. Because of this, in order to return an $\alpha$-suboptimal arm for \PIBAI, we must ensure that \textit{all arms that are more than $\alpha$-suboptimal are eliminated before the optimal arm is eliminated} (if it ever is). This is in contrast to \BAIspace in the classical multi-armed bandit setup: there, it is not problematic if the best (or even currently best) arm is eliminated at round $r$, so long as the best arm in consecutive rounds $r$ and $r+1$ changes at most by a small amount.\footnote{In particular, at most $2^{-r}\alpha$ for most of these aforementioned algorithms, since this implies that the final arm $\hat{I}$ is at most $\sum_{r=1}^{\infty} 2^{-r}\alpha = \alpha$-suboptimal.} We stress that this nuance is not merely a technicality but actually fundamental to the correctness proofs for \PIBAI.
\end{remark}

%% file: 4_median.tex

\section{Finding the Best Median}\label{sec:median}
The previous section provided algorithms for the general \PIBAIspace problem; in this section, these algorithms are adapted to the special case of the \CBAIspace problem, which is the focus of the paper. This requires an estimator of the median from contaminated samples that concentrates at an exponential rate\footnote{For the formal statement, see (ii) in the definition of \PIBAIspace in Section~\ref{sec:pibai}.}. Subsection~\ref{subsec:contaminated.estimation} develops these statistical results, which are of potential independent interest to the robust statistics community. Subsection~\ref{subsec:cbai-pac} then combines these results with the results of the previous section to conclude algorithms for \CBAI, and gives upper bounds on their sample complexity. Finally, Subsection~\ref{subsec:lb} gives information-theoretic lower bounds showing that the sample complexities of these algorithms are optimal (up to a small logarithmic factor).

\subsection{Estimation of Median from Contaminated Samples}\label{subsec:contaminated.estimation}
In this subsection, we develop the statistical results needed to obtain algorithms for the \CBAIspace problem. Specifically, we obtain tight, non-asymptotic sample-complexity bounds for median estimation from contaminated samples. We do this for all three adversarial models for the contamination (oblivious, prescient, and malicious). These results may be of independent interest to the robust statistics community.
\par The subsection is organized as follows. First, Subsection~\ref{subsubsec:med-finite} studies the concentration of the empirical median under the most general distributional assumptions and provides, for all three adversarial settings, tight upper and lower bounds on the sample complexity of median estimation from contaminated samples. Next, in Subsection~\ref{subsubsec:famnosig.estimation} we obtain more algorithmically useful guarantees by specializing these results to a family of distributions where the cdfs increase at least linearly in a neighborhood of the median, a very common assumption in the robust estimation literature. For this family of distributions, we explicitly compute for all three adversarial settings the unavoidable bias terms mentioned in Section~\ref{sec:cbai} for median estimation from contaminated samples. It is worth emphasizing that these results precisely quantify the effect of the three different adversarial strengths on the complexity of the problem of median estimation in the contamination model.
\par Throughout this subsection, we will only consider a single arm, and its true distribution will be denoted by $F$. For brevity of the main text, all proofs are deferred to Appendix~\ref{app:est}.

\subsubsection{General Concentration Results for Median Estimation}\label{subsubsec:med-finite}


Lemma~\ref{lem:change-med} implies that some control of the quantiles of $F$ is necessary to obtain guarantees for estimation of $\med(F)$ from contaminated samples. We begin by considering $F$ in the following family of distributions, which we stress makes the most general such assumption.

\begin{definition}\label{def:fam-quantile}
For any $\tmax \in (0, \half)$ and any non-decreasing 
function $R : [0, \tmax] \to \Real_{\geq 0}$, define $\Famquantile$ to be the family of all distributions $F$ satisfying 
\begin{align}
R(t) \geq 
\max\left\{
\QRF\left(\frac{1}{2} + t \right) - m, 
m - \QLF\left(\frac{1}{2} - t \right)
\right\}
\label{eq:def-R}
\end{align}
for any $t \in [0, \tmax]$ and any median $m \in m_1(F)$.
\end{definition}
In words, the function $R$ dictates the the maximal deviation from the median that a cdf $F$ can have for all quantiles in a small neighborhood $[\half - \bar{t}, \half + \bar{t}]$ around the median. Since $R$ is an arbitrary non-decreasing function, Definition~\ref{def:fam-quantile} gives the most general possible bound on these quantiles.

\par Note that if $\eps > \epsmax(\tmax) := \frac{2\tmax}{1+2\tmax}$, it is impossible to control the deviation of the contaminated median from the true median for $F \in  \Famquantile$. Indeed, the requirement $\eps \leq \epsmax(\tmax) := \frac{2\tmax}{1+2\tmax}$ is equivalent to $\tmax \geq \frac{\eps}{2(1-\eps)}$, which, by Lemma~\ref{lem:change-med}, is the largest possible deviation from the $\half$-quantile.

Combining Lemma~\ref{lem:change-med} with the definition of $R$ in Definition~\ref{def:fam-quantile} immediately yields the following tight bound on how far the contaminated median can be moved from the true median for any $F\in\Famquantile$.
\begin{corollary}\label{corol:change-med-R}
For any $\tmax \in (0, \half)$, any $\eps \in (0, \epsmax(\tmax))$, and any non-decreasing function $R : [0, \tmax] \to \Real_{\geq 0}$,
\[
\sup_{\substack{F \in \Famquantile,\, m \in \med(F),\\
\mathrm{distribution}\; G, \, \tilde{m} \in \med((1 - \eps) F + \eps G)}}
\left|\tilde{m} - m \right|
= R\left(
\frac{\eps}{2(1-\eps)}
\right)
\]
\end{corollary}

We now turn to estimation results. At this point it is necessary to distinguish between the three adversarial settings for the contamination. We begin with guarantees for the oblivious and prescient adversarial settings, which -- somewhat surprisingly -- turn out to have the same rates. In particular, we show that with probability at least $1 - \delta$, the estimation error is bounded above by $R(\tfrac{\eps}{2(1 - \eps)} + O(\sqrt{\tfrac{\log 1/\delta}{n}}))$. Note that if $R$ is Lipschitz, then this quantity is bounded above by the unavoidable uncertainty term $R\left(\tfrac{\eps}{2(1-\eps)}\right)$ (see Corollary~\ref{corol:change-med-R}) plus an error term that decays quickly with $n^{-1/2}$ rate and has sub-Gaussian tails in $\delta$. This confidence-interval term is of optimal order since it is tight even for estimating the mean (and thus median) of a Gaussian random variable with known unit variance from \emph{uncontaminated} samples.

\begin{lemma}\label{lem:est-med-R}
Let $\tmax \in (0, \half)$, $\eps \in (0, \epsmax(\tmax))$, and $F \in \Famquantile$. Let $Y_i \sim F$ and $D_i \sim \Ber(\eps) $, for $i \in [n]$, all be drawn independently. Let $\{Z_i\}_{i \in [n]}$ be arbitrary random variables possibly depending on $\{Y_i, D_i\}_{i \in [n]}$, and define $X_i = (1 - D_i)Y_i + D_iZ_i$. Then, for any confidence level $\delta \in (0,1)$ and number of samples
$n \geq 2 \left(\tmax - \tfrac{\eps}{2(1-\eps)}\right)^{-2} \logtdel$
, we have 
\[
\Prob\left(
\sup_{m \in \med(F)}
\left| \hatmed(X_1, \dots, X_n) - m \right| \leq
R\left(\frac{\eps}{2(1 - \eps)} + \sqrt{\frac{2\log (2/\delta)}{n}} \right)
\right)
\geq 1 - \delta.
\]
\end{lemma}
Note that because the class $\Famquantile$ only assumes control on the $[\half \plusminus \tmax]$ quantiles, the minimum sample complexity $n$ must grow as $\tmax$ approaches $\frac{\varepsilon}{2(1-\varepsilon)}$, since by Lemma~\ref{lem:change-med} this is the largest quantile deviation that the contaminated median can be moved from the true median.

We now turn to malicious adversaries and derive analogous tight, non-asymptotic sample complexity bounds for median estimation from contaminated samples. We show that it possible to obtain estimation accuracy of $R(\eps)$ in this malicious adversarial setting (Lemma~\ref{lem:est-med-malicious-R}). Note that this is weaker than the the accuracy of $R(\tfrac{\eps}{2(1 - \eps)})$ obtained above against oblivious and prescient adversaries. However, we also show that this dependency is tight and unavoidable (Lemma~\ref{lem:est-med-malicious-R-tight}). Moreover, the $O(\sqrt{\tfrac{\log 1/\delta}{n}})$ error term in Lemma~\ref{lem:est-med-malicious-R} is tight for the same reason as it was in Lemma~\ref{lem:est-med-R}; see the discussion there. Finally, we remark that the upper bound on $\eps$ and the lower bound on the sample complexity $n$ in Lemma~\ref{lem:est-med-malicious-R} are exactly the analogues of the corresponding bounds in Lemma~\ref{lem:est-med-R}; the only difference is that malicious adversaries can force the contaminated distributions to have medians at roughly $F^{-1}(\half \plusminus \eps)$, resulting in our need for control of these further quantiles.

\begin{lemma}\label{lem:est-med-malicious-R}
Let $\tmax \in (0, \half)$, $\eps \in (0, \tmax)$, and $F \in \Famquantile$. Let $(Y_i,D_i)$, for $i \in [n]$, be drawn independently with marginals $Y_i \sim F$ and $D_i \sim \Ber(\eps)$. Let $\{Z_i\}_{i \in [n]}$ be arbitrary random variables possibly depending on $\{Y_i, D_i\}_{i \in [n]}$, and define  $X_i = (1 - D_i)Y_i + D_iZ_i$. Then for any confidence level $\delta \in (0,1)$ and number of samples
$n \geq  2 \left(\tmax -\eps\right)^{-2} \log\tfrac{3}{\delta}$,
\[
\Prob\left(
\sup_{m \in \med(F)}
\left| \hatmed(X_1, \dots, X_n) - m \right| \leq
R\left(\eps  + \sqrt{\frac{2\log(3/\delta)}{n}} \right)
\right)
\geq 1 - \delta.
\]
\end{lemma}

\begin{lemma}\label{lem:est-med-malicious-R-tight}
Let $\tmax \in (0, \half)$, $\eps \in (0, \tmax)$, $\delta \in (0,1)$, $n \geq \half\eps^{-2}\logdel$, and $R : [0, \tmax] \to \Real_{\geq 0}$ be any strictly increasing function. Then there exists a distribution $F \in \Famquantile$ and a joint distribution on $(D, Y, Z)$ with marginals $D \sim \Ber(\eps)$ and $Y \sim F$, such that $F$ has unique median and
\[
\Prob\left(
\left| \hatmed(X_1, \dots, X_n) - \med(F) \right| \geq
R\left(\eps  - \sqrt{\frac{\log(1/\delta)}{2n}} \right)
\right)
\geq 1 - \delta,
\]
where each $\{(D_i, Y_i, Z_i)\}_{i \in [n]}$ is drawn independently from the joint distribution, and $X_i := (1 - D_i)Y_i + D_i Z_i$.
\end{lemma}

\subsubsection{Improved Concentration under Quantile Control}\label{subsubsec:famnosig.estimation}
Note that if a cdf $F$ changes little around its median, then the neighboring quantiles are harder to distinguish from each other, and thus it is more difficult to estimate the median. In order to obtain more algorithmically useful rates for median estimation from contaminated samples, we now introduce a more specific class of cdfs $F$ that increase at least linearly in a neighborhood around the median. This ensures that $F$ is not ``too flat'' in this neighborhood, which we stress is a very standard assumption for median estimation.

\begin{definition}\label{def:2-param}
For any $\tmax \in (0, \half)$ and $B > 0$, let $\Famnosig$ be the family of distributions $F$ that have a unique median and satisfy
\begin{align}
|F(x_1) - F(x_2)| \geq \frac{1}{B\sigmed(F)}|x_1 - x_2|
\label{eq:def-Fam}
\end{align}
for all $x_1, x_2 \in  \left[\QLF(\half - \tmax), \QRF(\half + \tmax)\right]$.
\end{definition}

We make a few remarks about the definition. i) Requiring the right-hand side of~\eqref{eq:def-Fam} to scale inversely in the median absolute deviation (MAD) $m_2(F)$ ensures closure of $\Famnosig$ under scaling; see below. We also mention that $m_2(F)$ is a robust measure of the spread of $F$ (it is the ``median moment'' analogue of variance), and controls the width of the median's unidentifiability region. ii) If $F \in \Famnosig$, then $F \in \Famquantile$ for $R(t) = Bm_2(F)t$. iii) Distributions in $\Famnosig$ are not required to have densities, nor even be continuous. iv) The family $\Famnosig$ has many natural and expected properties, such as closure under scaling and translation. These properties are gathered in Lemma~\ref{lem:fam-properties}, which is deferred to Appendix~\ref{app:prop-fam} for brevity of the main text.

\begin{remark} [Examples]
  Most common distributions belong to $\Famnosig$ for some values of the parameters $\tmax$ and $B$. Moreover, by Lemma \ref{lem:fam-properties}, if a distribution is in $\Famnosig$, then all scaled and translated versions are as well. A short list includes: i) the Gaussian distribution, for any $\tmax \in (0, \half)$ and any $\DS B\geq \frac{q_{3/4}}{\phi(q_{1/2+\bar t})}$, where $\phi$ is the standard Gaussian density and $q_\alpha$ is the corresponding $\alpha$-quantile; ii) the uniform distribution on any interval, for any $\bar t\in (0,\half)$ and any $B\geq4$; and iii)  any distribution $F$ with positive density $F'$, for any $\tmax \in (0,\half)$ and any $B\geq \left(m_2(F)  \min_{x\in [\QLF(\half - \tmax), \QRF(\half + \tmax)]}F'(x)\right)^{-1}$.
\end{remark}

We now apply the estimation results for $\Famquantile$ to estimation for $\Famnosig$. By Lemma~\ref{lem:fam-properties}, whenever $F \in \Famnosig$, then $F \in \Famquantile$ for
$R(t) = t Bm_2(F)$. Corollary~\ref{corol:change-med-R} immediately yields a bound on the size of the contamination region in terms of the quantity 
\[
\Uncertaintygen := Bm_2\frac{\eps}{2(1-\eps)}.
\]
\begin{corollary}\label{corol:change-med}
For any $\tmax \in (0, \half)$, any $\eps \in (0, \epsmax(\tmax))$, and any $F \in \Famnosig$,
\[
\sup_{\substack{\mathrm{distribution}\; G, \\ \tilde{m} \in \med((1 - \eps) F + \eps G)}}
\left|\tilde{m} - \med(F) \right|
\leq
\UncertaintyF.
\]
\end{corollary}

\begin{remark}[Tightness of Corollary~\ref{corol:change-med}] \label{Remark1}
For any $a > 0$, let $F$ be the uniform distribution over $[0,a]$. Then $\med(F) = \tfrac{a}{2}$ and $\sigmed(F) = \tfrac{a}{4}$, and thus $F \in \Famnosig$ for any $\bar t \in (0, \half)$ and $B = 4$. Now for any $\eps \in (0, \epsmax(\tmax)) \subset (0, \half)$, let $G$ be the uniform distribution over $[a, \tfrac{a}{1-\eps}]$. Then $\tilde{F} = (1 - \eps)F + \eps G$ is the uniform distribution over $[0, \tfrac{a}{1-\eps}]$ and has median $m_1(F) = \tfrac{a}{2} + \tfrac{\eps}{2(1-\eps)}a = \med(F) + \UncertaintyF$. An identical argument on $F(-\cdot)$ and $G(-\cdot)$ yields an example where the median decreases by $-\UncertaintyF$, so the bound in Corollary~\ref{corol:change-med} is tight.
\end{remark}

Now we turn to median estimation bounds for $\Famnosig$. Since $R(t) = tBm_2(F)$ is clearly Lipschitz, the discussion preceding Lemmas~\ref{lem:est-med-R} and~\ref{lem:est-med-malicious-R} about the optimality of the following error bounds applies. That is, the error decomposes into the sum of the unavoidable uncertainty term $\UncertaintyF$ plus a confidence-interval term decaying at a $n^{-1/2}$ rate and with sub-Gaussian tails in $\delta$, both of which are optimal. The $\bar{t}$ terms in the sample complexity bounds appear for the analogous reasons as in Lemma~\ref{lem:est-med-R}; see the discussion there. For brevity, we omit the proofs of the following corollaries since they follow immediately from Lemmas~\ref{lem:est-med-R},~\ref{lem:est-med-malicious-R}, and~\ref{lem:fam-properties}. 
\begin{corollary}\label{corol:est-med}
  Consider the same setup as in Lemma~\ref{lem:est-med-R}, except with the added restriction that $F \in \Famnosig$. Then, for any confidence level $\delta \in (0,1)$, error level $E > 0$, and number of samples $n \geq 2 \max\left(\tfrac{B^2\sigmedsq(F)}{E^2}, \,  \left(\tmax - \tfrac{\eps}{2(1-\eps)}\right)^{-2} \right) \logtdel$, we have
\[
\Prob\Bigg(
\left| \hatmed(X_1, \dots, X_n) - \med(F) \right| \leq
\UncertaintyF + E\Bigg)
\geq 1 - \delta.
\]
\end{corollary}

\begin{corollary}\label{corol:est-med-malicious}
Consider the same setup as in Lemma~\ref{lem:est-med-malicious-R}, except with the added restriction that $F \in \Famnosig$. Then for any confidence level $\delta \in (0,1)$, error level $E > 0$, and number of samples
$n \geq  2 \max\left(\tfrac{B^2\sigmedsq(F)}{E^2},\,\left(\tmax -\eps\right)^{-2} \right) \log\tfrac{3}{\delta}$,
\[
\Prob\Bigg(
\left| \hatmed(X_1, \dots, X_n) - \med(F) \right| \leq
\UncertaintyFprescient + E
\Bigg)
\geq 1 - \delta.
\]
\end{corollary}
In Corollary~\ref{corol:est-med-malicious} above, we have defined \[
\Uncertaintygenprescient := Bm_2\eps,
\] which is a tight bound on the uncertainty in median estimation that a malicious adversary can induce (see Lemma~\ref{lem:est-med-malicious-R-tight}).

\subsection{Algorithms}\label{subsec:cbai-pac}
In this subsection, we apply the algorithms developed in Section~\ref{sec:pibai} for the general problem of \PIBAIspace to the special case of \CBAI. In order to implement the estimator required by part (ii) of the definition of \PIBAIspace (see Section~\ref{sec:pibai}), we make use of the guarantees proved in Subsection~\ref{subsec:contaminated.estimation} for median estimation from contaminated samples. Nearly matching information-theoretic lower bounds are provided in the following subsection and show that the algorithms below have optimal sample complexity (up to a small logarithmic factor).
\par We will only assume that the arms' distributions are in $\Famnosig$ and have robust second moments uniformly bounded\footnote{It is typically necessary to have bounds on higher order moments in order to control the error of estimating lower order moments, see e.g.~\citep{bubeck2013bandits}.} by some $\bar{m}_2$; this will allow us to perform median estimation from contaminated samples. For simplicity, the parameters of this family are assumed to be known to the algorithm beforehand\footnote{The classes of distributions that we define have two parameters that would be hard or even impossible to estimate in practice, namely $B$ and $\bar m_2$. In classical setups in which algorithms need to learn means instead of medians, it is common to assume that there is a second moment $\sigma^2$ which is known, or at least bounded by a known constant. This is because in practice $\sigma^2$ is hard to estimate from the data. Note that if the distributions $F_i$ are known to have bounded support, $\bar m_2$ can be chosen as the size of the domain. It seems that knowing $B$ is the price to pay to go from the classical to the robust setup. Indeed, even in the non-contaminated case, it is well known that the median of a distribution can be estimated at a parametric rate only when the distribution is not too flat around its median. In that case, the estimation rate depends on how ``\textit{non-flat}'' the distribution is, which is exactly $B m_2(F)$ in the present setup. In a parametric setup, i.e., when the distributions of the arms $F_i, i=1,\ldots,K$ are known up to some location and/or scale parameters, $B$ would be known since it is translation and scale invariant (see Lemma~\ref{lem:fam-properties}).}. Note, however, that it is not necessary to know $\eps$, but only an upper bound $\eps_0$ on $\eps$; indeed $\eps$ can be replaced with $\eps_0$ in all our algorithms. It is very standard in the robust statistics literature to assume knowledge of such an upper bound $\eps_0$ on the amount of contamination. The parameter $\bar t$ can then be arbitrarily chosen by the algorithms, so long as it is bigger than $\eps_0/(2(1-\eps_0))$ (or $\eps_0$ in the malicious case) and smaller than 1/2. 
\par The sample complexities we prove below are in terms of the \emph{effective gaps} $\tilde\Delta_i := (m_1(F_{i^*}) - U_{i^*}) - (m_1(F_{i}) + U_i)$ of the suboptimal arms $i \neq i^*$, where $i^* := \argmax_{i \in [k]} m_1(F_i)$ is the arm with the highest median. Here $U_i$ is the unavoidable uncertainty term for median estimation from contaminated samples under the given adversarial setting. In Subsection~\ref{subsec:contaminated.estimation}, we explicitly computed $U_i = U_{\eps, B, m_2(F_i)}$ for oblivious and prescient adversaries (Corollary~\ref{corol:est-med}) and $U_i = U_{\eps,B,\sigmed(F_i)}^{\textsc{(malicious)}}$ for malicious adversaries (Corollary~\ref{corol:est-med-malicious}). We emphasize that the strength of the adversarial setting is encapsulated in the corresponding $U_i$.
\par First, we discuss how the simple Algorithm \ref{alg:Simple} yields an $(\alpha,\delta)$-PAC algorithm for \CBAIspace without modification. In this setting, $\hat p_i$ is the empirical median of the payoffs of arm $i$. The quantity $n_{\alpha,\delta}$ is equal to $2 \max(\tfrac{B^2\bar m_2^2}{\alpha^2}, \,  (\tmax - \tfrac{\eps}{2(1-\eps)})^{-2} ) \logtdel$ against oblivious and prescient adversaries (Corollary \ref{corol:est-med}) and is equal to $2 \max(\tfrac{B^2\bar m_2^2}{\alpha^2}, \,  \left(\tmax - \eps\right)^{-2})  \log\tfrac{3}{\delta}$ against malicious adversaries (Corollary~\ref{corol:est-med-malicious}). Theorem~\ref{Thm:Simple} along with the observation that the constant term in the sample complexity $n_{\alpha,\delta}$ only introduces a term that is negligible in the overall sample complexity of Algorithm~\ref{alg:Simple}, immediately yields the following result.

\begin{theorem}\label{Thm:SimpleCBAI}
Let $F_i \in \Famnosig$ with $m_2(F) \leq \bar{m}_2$ for each arm $i \in [k]$, and let the adversary be oblivious, prescient, or malicious. For any $\alpha>0$ and $\delta\in (0,1)$, Algorithm \ref{alg:Simple} is an $(\alpha,\delta)$-PAC \CBAIspace algorithm with sample complexity $O(kn_{\alpha/2,\delta/k})=O\left(\frac{k}{\alpha^2}\log\frac{k}{\delta}\right)$. 
\end{theorem}

Algorithm~\ref{alg:Simple} does not adapt to the difficulty of the problem instance. A natural approach is to try to apply the Successive Elimination Algorithm to \BAIspace directly. Unfortunately, this algorithm needs concentration inequalities even for small sample sizes, whereas our guarantees are only valid when the sample size is greater than a certain threshold.
Hence, we slightly modify the Successive Elimination Algorithm to obtain additional samples in an initial exploration phase; the number of samples required in this initial exploration phase is dictated by our median estimation results form Subsection~\ref{subsec:contaminated.estimation}. We similarly modify later rounds to take additional samples in order to produce updated median estimates at the desired certainty levels. Pseudocode is given in Algorithm \ref{alg:SuccElimCBAI}.

\begin{algorithm}[t]
$S \leftarrow [k]$, $r \leftarrow 1$
\\ Sample each arm $\lceil N\log(\tfrac{\pi^2 k}{2\delta})\rceil$ times
\\ \While{$|S| > 1$}{
Sample each arm $i \in S$ for 
$1 + \lceil 2N\log(\tfrac{r+1}{r}) \rceil $
times and produce $\hat{p}_{i,r}$ from all past samples\\
$S \leftarrow \{i\in S\,:\, \hat p_{i,r} \geq \max_{j\in S}\hat p_{j,r} - 2\alpha_{r, 6\delta/(\pi^2 k r^2)} \}$\\
$r \leftarrow r+1$
}
Output the only arm left in $S$
\caption{Adaptation of successive elimination algorithm (see Algorithm~\ref{alg:SuccElim}) for $\CBAI$. For oblivious and prescient adversaries, $N:= 2(\tmax - \tfrac{\eps}{2(1-\eps)})^{-2}$; and for malicious adversaries $N:= 2(\tmax - \eps)^{-2}$. In all settings, $\alpha_{r,\delta} := \sqrt{\tfrac{2B\bar m_2^2\log \frac{3}{\delta}}{r}}$.}
\label{alg:SuccElimCBAI}
\end{algorithm}

\begin{theorem} \label{Thm:SuccElimCBAI}
Let $\delta\in (0,1)$, let $F_i \in \Famnosig$ with $m_2(F) \leq \bar{m}_2$ for each arm $i \in [k]$, and let the adversary be either oblivious, prescient, or malicious. With probability at least $1-\delta$, Algorithm \ref{alg:SuccElimCBAI} outputs the optimal arm after using at most 
$\tilde{O}\left(\sum_{i\neq i^*} \left( \tfrac{1}{\tilde\Delta_i^2} + N\right) \log\left(\frac{k}{\delta\tilde\Delta_i}\right)\right)$
samples.
\end{theorem}

The proof follows from Theorem~\ref{Thm:SuccElim} and is deferred to Appendix~\ref{app:algorithms}. Like Algorithm~\ref{alg:SuccElim}, Algorithm~\ref{alg:SuccElimCBAI} can be stopped earlier in order to obtain an $(\alpha,\delta)$-PAC guarantee.


\subsection{Lower Bounds}\label{subsec:lb}
Here we provide information-theoretic lower bounds on the sample complexity of the \CBAIspace problem that match, up to small logarithmic factors, the algorithmic upper bounds proved above in Subsection~\ref{subsec:cbai-pac} for each of the three adversarial settings. The main insight is that we can reduce hard instances of the \BAIspace problem to hard instances of the \CBAIspace problem. In this way, we can leverage the sophisticated lower bounds already developed in the classical multi-armed bandit literature. 

Let $i^* := \argmax_{i \in [k]} m_1(F_{i})$ denote the optimal arm. As in Section~\ref{sec:pibai}, $\tilde \Delta_i := (m_1(F_{i^*}) - U_{i^*}) - (m_1(F_{i}) + U_i)$ denotes the effective gap for suboptimal arms $i \neq i^*$, and $U_i$ denotes the unavoidable uncertainty term for median estimation. We emphasize that since the power of the adversary is encapsulated in the $U_i$, and therefore also the effective gaps $\tilde{\Delta}_i$, we can address all three adversarial settings simultaneously by proving a lower bound in terms of the effective gaps.

In Section~\ref{sec:pibai}, we argued that the \PIBAIspace problem is impossible when there exists suboptimal arms with non-positive effective gaps; the \CBAIspace problem also has this property. Indeed, if $i \neq i^*$ satisfies $\tilde\Delta_i \leq 0$, then there exist distributions $G_{i^*}$ and $G_i$ such that the resulting contaminated distributions $\tilde{F}_i := (1 - \eps)F_i + \eps G_{i}$ and $\tilde{F}_{i^*} := (1 - \eps) F_{i^*} + \eps G_{i^*}$ have equal medians. Moreover, since the distributions $F$ are arbitrary (\CBAIspace makes no parameteric assumptions), it is impossible to determine whether arm $i$ or $i^*$ has a higher true median. Thus, any \CBAIspace algorithm, even with infinite samples, cannot succeed with probability more than $\half$. Therefore, we only consider the setting where all effective gaps $\tilde\Delta_i$ are strictly positive.
\par The results in this section lower bound the sample complexity of any \CBAIspace algorithm in terms of the effective gaps. We will focus on lower bounds for the function class $\Famnosig$ and provide lower bounds matching the upper bounds for \CBAIspace in Section~\ref{subsec:cbai-pac}.

\begin{theorem}\label{thm:lb-pac}
  Consider \CBAIspace against an oblivious, prescient, or malicious adversary. There exists a positive constant $B$ such that for any number of arms $k \geq 2$, any confidence level $\delta \in (0, \tfrac{3}{20})$, any suboptimality level $\alpha \in (0, \tfrac 1 6)$, any contamination level $\eps \in (0, \tfrac{1}{15})$, any regularity level $\tmax \in (0, \tfrac{1}{10})$, and any effective gaps $\{\tilde{\Delta}_i\}_{i \in [k] \setminus i^*} \subset (0,\tfrac{1}{3})$, there exists a \CBAIspace instance with $F_1, \dots, F_k \in \mathcal{F}_{\tmax, B}$ for which any $(\alpha,\delta)$-PAC algorithm has expected sample complexity
\[
\E[T] = \Omega\left(
\sum_{i \in [k] \setminus \{i^*\}} \frac{1}{\max(\tilde{\Delta}_i, \alpha)^2}\logdel 
\right),
\]
where $i^* = \argmax_{i \in [k]} m_1(F_i)$ is the optimal arm.
\end{theorem}
Taking the limit as $\alpha \to 0$ in Theorem~\ref{thm:lb-pac} immediately yields the following lower bound on $(0,\delta)$-PAC algorithms.
\begin{corollary}\label{corol:lb}
Consider the same setup as in Theorem~\ref{thm:lb-pac}. There exists a \CBAIspace instance $F_1, \dots, F_k \in \mathcal{F}_{\tmax, B}$ for which any $(0,\delta)$-PAC algorithm has expected sample complexity
\[
\E[T] = \Omega\left(
\sum_{i \in [k] \setminus \{i^*\}} \frac{1}{\tilde{\Delta}_i^2}\logdel
\right).
\]
\end{corollary}

\paragraph*{Proof Sketches.}
We now sketch the proofs of Theorem~\ref{thm:lb-pac} and Corollary~\ref{corol:lb}. Full details are deferred to Appendix~\ref{app:lb}.remove{ for brevity of the main text.}

The key idea in the proof is to ``lift''  hard \BAIspace instances to hard \CBAIspace problem instances. Specifically, let $\{P_i\}_{i \in [k]}$ be distributions for arms in a \BAIspace problem, let $\{U_i\}_{i \in [k]}$ be the corresponding unavoidable uncertainties for median estimation of $P_i$ in \CBAIspace, and let $i^*$ denote the best arm. We say $\{P_i\}_{i \in [k]}$ admits an ``$(\eps, \Famnosig)$\emph{CBAI-lifting}'' to $\{F_i\}_{i \in [k]}$ if (i) $F_i \in \Famnosig$ for each $i$, (ii) there exists some adversarial distributions $\{G_i\}_{i \in [k]}$ such that each $P_i= (1 - \eps) F_i + \eps G_i$; and (iii) the effective gaps $\tilde{\Delta}_i$ for the $\{F_i\}_{i \in [k]}$ are equal to the gaps $\Delta_i$ in the original \BAIspace problem. Intuitively, such a lifting can be thought of as choosing $F_{i^*}$ to have the smallest possible median and the other $F_i$ to have the largest possible median, which still being consistent with $P_i$ and remaining in the uncertainty region corresponding to the $U_i$.

\par Armed with this informal definition, we now outline the central idea behind the reduction. Let $\{P_i\}_{i \in [k]}$ be a ``hard'' \BAIspace instance with best arm $i^*$, and assume it admits an $(\eps,\Famnosig)$-\CBAI-lifting to $\{F_i\}_{i \in [k]}$. Consider running a \CBAIspace algorithm $\calA$ on the samples obtained from  $\{P_i\}$; we claim that if $\calA$ is $(\alpha,\delta)$-PAC, then it must output an $\alpha$-suboptimal arm for the original \BAIspace problem. Indeed, by (ii), the samples from the \BAIspace instance have the same law as the samples that would be obtained from the \CBAIspace problem with $\{F_i\}$, and by (iii), the effective gaps in the \CBAIspace problem are equal to the gaps in the original problem. Therefore, $\calA$ must have sample complexity for this problem that is no smaller than the sample complexity of the best \BAIspace algorithm.

\par Note that the $\{F_i\}$ in the above lifting need only \emph{exist}, as we do not explicitly use these distributions but instead simply run $\calA$ on samples from the original bandit instance. Ensuring the existence of such a lifting is the main obstacle, due the following technical yet important nuance: most \BAIspace lower bounds are constructed from Bernoulli or Gaussian distributions, neither of which is compatible with the reduction described above. The problem with Bernoulli arms is that they do not have liftings to $\Famnosig$: any resulting $F_i$ would not satisfy~\eqref{eq:def-Fam} and thus cannot be in $\Famnosig$.  Gaussian arms  run into a different problem: it is not clear how to shift Gaussian arms up or down far enough to change the median by exactly the maximum uncertainty amount $U_i$. We overcome this nuance by considering arms with \textit{smoothed Bernoulli distribution} $\SBer(p)$, which we define to be the uniform mixture between a Bernoulli distribution with parameter $p$ and a uniform distribution over $[0,1]$. Indeed, this distribution is $(\eps, \Famnosig)$-\CBAI-liftable: unlike the Bernoulli distribution, it is smooth enough to have liftings in $\Famnosig$; and unlike the Gaussian distribution, the median of the appropriate lifting of $\SBer(p)$ is exactly $U_i$ away from the median of $\SBer(p)$. Both of these facts are simple calculations; see Appendix~\ref{app:lb} for details.

The only ingredient remaining in the proof is to prove that there are hard instances for \BAIspace with $\SBer$-distributed arms, which follows easily from Lemma 1 and Remark 5 in \cite{kaufmann2016complexity}.

%

%% file: 5_quality.tex

\section{Quality Guarantees for the Selected Arm}\label{sec:quality}
In the previous Section~\ref{sec:median}, we developed algorithms for \CBAIspace and proved that they select the best (or an approximately best) arm with high probability. However, in many applications, it is desirable to also provide guarantees on the quality of the selected arm such as ``with probability at least $80\%$, a new random variable $Y \sim F_{\hat{I}}$ is at least 10.''

Such a quality guarantee could be accomplished directly using the machinery developed earlier in Subsection~\ref{subsec:contaminated.estimation} by, for example, estimating the $0.2$ quantile in addition to the median (the $0.5$ quantile). However, this approach has two problems. First, if we would like to obtain such a guarantee for multiple probability levels (such as $55\%,60\%,\dots,$ etc.) we would have to perform a separate estimation for each of the corresponding quantiles. Second, and more importantly, estimation of a quantile $q$ from contaminated samples requires control of quantiles in a $\frac{\eps}{2(1 -\eps)}$ neighborhood of $q$ (this follows by an identical argument as in Lemma~\ref{lem:change-med}), which can severely restrict the range of quantiles that are  possible to estimate.

We circument both of these problems by estimating the \textit{median absolute deviation} (MAD) in addition to the median. The MAD describes the scale of the tails of a distribution away from the median and is the appropriate analogue to variance (which may not exist for $F$ and, we stress, is not estimable from contaminated samples) for this contamination model setting.

This section is organized as follows. Subsection~\ref{subsec:mad-finite} develops these statistical results for estimation of the MAD from contaminated samples. Subsection~\ref{subsec:cbai-quantile} then uses these results to prove that the \CBAIspace algorithms presented earlier in Subsection~\ref{subsec:cbai-pac} -- with no additional modifications -- can also provide quality guarantees for the selected arm. That is, we show that in addition to outputting the best (or an approximately best) arm, these algorithms also provide quality guarantees to a certain precision ``for free'' without needing extra samples.

\subsection{Estimation of Second Robust Moment from Contaminated Samples}\label{subsec:mad-finite}
In this subsection, we develop the statistical results needed to obtain quality guarantees for the selected arm in the \CBAIspace problem. Specifically, we consider only a single arm and obtain non-asymptotic sample-complexity bounds for estimation of the second robust moment (the MAD) from contaminated samples. We do this for all three adversarial models for the contamination (oblivious, prescient, and malicious). These results may be of independent interest to the robust statistics community.
\par The subsection is organized as follows. First, Subsection~\ref{subsubsec:mad:class} introduces a class of distributions for which the MAD is estimable from contaminated samples. The few assumptions we make are common in the statistics literature, and we give multiple examples showing that many common distributions are included in this class. Subsection~\ref{subsubsec:mad:estimate} then proves finite-sample guarantees for MAD estimation from contaminated samples for this class of distributions and for all three adversarial contamination settings.

\subsubsection{A Class of Distributions with Estimable Second Robust Moment}\label{subsubsec:mad:class}
Like the median, the MAD is not fully identifiable from contaminated samples (see the discussion in Section~\ref{sec:median-infinite}). However, we can estimate the MAD up to a reasonable region of uncertainty with only a few additional assumptions.

\begin{definition}\label{def:fam-mad}
For any $\tmax \in (0, \half)$, $B > 0$, $\sigmedmax > 0$, and $\kappa \geq 0$, let $\Fam$ be the family of distributions $F$ with unique robust moments $m_1(F)$, $m_2(F)$, and $m_4(F)$ satisfying:
\begin{itemize}
\item[(i)]\eqref{eq:def-Fam} holds for all $x_1, x_2 \in  \left[\QLF(\half - \tmax), \QRF(\half + \tmax)\right] \cup [m_1(F) \plusminus 2m_2(F)]$.
\item[(ii)] $\sigmed(F) \leq \sigmedmax$.
\item[(iii)] $m_2(F) \leq \kappa m_4(F)$.
\end{itemize}
\end{definition}
Let us make a few remarks on this definition.
i) We emphasize that our \CBAIspace algorithms already have optimality guarantees when the arm distributions are in $\Famnosig$ (see Section~\ref{subsec:cbai-pac}); the additional assumption that the arm distributions are in $\Fam$ allows us to also obtain quantile guarantees on the arm returned by the algorithm (see Section~\ref{subsec:cbai-quantile}).
ii) Property (iii) of the definition requires a bound on the fourth robust moment, the strength of which is dictated by the parameter $\kappa$. Informally, higher robust moments are necessary to control the error of estimation of lower robust moments, analogous to how bounds on variance (resp. kurtosis) are typically necessary for estimation of a distribution's mean (resp. variance). 
iii) Note that distributions in $\Fam$ are not required to have densities, nor even be continuous.
iv) The family $\Fam$ satisfies some natural and expected properties, such as closure under translation. These properties are gathered and proved in Lemma~\ref{lem:fam-properties}, which is deferred to Appendix~\ref{app:prop-fam} for brevity of the main text.

\begin{remark} [Examples]
Many common distributions (e.g. the normal distribution, the uniform distribution, the Cauchy distribution, etc.) belong to $\Fam$ for some values of the parameters $\tmax, B, \bar{m}_2$, and $\kappa$. For instance, the uniform distribution on the interval $[a,b]$ is in $\Fam$ for any $\bar t\in (0,\half)$, any $B \geq 4$, any $\kappa \geq 2$, and any $\bar{m}_2 \geq \tfrac{b-a}{4}$. Another example is that the Cauchy distribution with scale parameter $a$ belongs to $\Fam$ for any $\bar t\in (0,\half)$, any  $B\geq \pi(1+\max(1,\tan(\pi \bar t)^2))$, and $\kappa\geq \sqrt 3 -1$, and any $\bar m_2 \geq a$. We recall that a Cauchy distribution with scale parameter $a>0$ has density $\DS \frac{1}{\pi a}\frac{1}{1+\left(\frac{x-x_0}{a}\right)^2}, x\in\Real$, where $x_0\in \Real$ is a location parameter. For such a distribution, it is easy to check that $m_1(F)=x_0$, $m_2(F)=a$ and $m_4(F)=a(\sqrt 3 -1)$. Note that neither the mean nor the variance exists for the Cauchy distribution.
\end{remark}

\subsubsection{Concentration Results for Estimation of Second Robust Moment}\label{subsubsec:mad:estimate}
We are now ready to present finite-sample guarantees on the speed of convergence of the empirical MAD $\hat{m}_2$ to the underlying distribution's MAD $m_2$ when given contaminated samples. For brevity, proofs are deferred to Appendix~\ref{app:est:mad}.
\par We first present results for the oblivious and prescient adversarial settings. As before, the estimation error decomposes into the sum of two terms: a bias term reflecting the uncertainty the adversary can inject given her contamination level, and a confidence-interval term that shrinks with an optimal $n^{-1/2}$ rate and has sub-Gaussian tails in $\delta$.

\begin{lemma}\label{lem:est-mad}
Let $\tmax \in (0, \half)$, $\eps \in (0, \min(\epsmax(\tmax), \tfrac{1}{B}))$, and $F \in \Fam$. Let $Y_i \sim F$ and $D_i \sim \Ber(\eps) $, for $i \in [n]$, all be drawn independently. Let $\{Z_i\}_{i \in [n]}$ be arbitrary random variables possibly depending on $\{Y_i, D_i\}_{i \in [n]}$, and define $X_i = (1 - D_i)Y_i + D_iZ_i$. Then, for any confidence level $\delta > 0$, error level $E > 0$, and number of samples $n \geq 2 \max\left(
\tfrac{16\kappa^2B^2\sigmedmax^2}{E^2}, \left(\min(\tmax, \tfrac{1}{B}) - \tfrac{\eps}{2(1-\eps)}\right)^{-2}\right) \log \frac{4}{\delta}$,
\[
\Prob\Bigg(
\left| \hatsigmed(X_1, \dots, X_n) - \sigmed(F) \right|
\leq
(1 + 2\kappa)\UncertaintyF + E
\Bigg)
\geq 1 - \delta.
\]
\end{lemma}

Similar to median estimation above, it is possible to estimate the MAD even in the malicious adversarial setting. 

\begin{lemma}\label{lem:est-mad-malicious}
Let $\tmax \in (0, \half)$, $\eps \in (0, \min(\tmax, \tfrac{1}{B}))$, and $F \in \Fam$. Let $(Y_i,D_i)$, for $i \in [n]$, be drawn independently with marginals $Y_i \sim F$ and $D_i \sim \Ber(\eps)$. Let $\{Z_i\}_{i \in [n]}$ be arbitrary random variables possibly depending on $\{Y_i, D_i\}_{i \in [n]}$, and define  $X_i = (1 - D_i)Y_i + D_iZ_i$. Then, for any confidence level $\delta > 0$, error level $E > 0$, and number of samples $n \geq 2 \max\left(
\tfrac{16\kappa^2B^2\sigmedmax^2}{E^2}, \left(\min(\tmax, \tfrac{1}{B}) - \eps\right)^{-2}\right) \log \frac{6}{\delta}$,
\[
\Prob\Bigg(
\left| \hatsigmed(X_1, \dots, X_n) - \sigmed(F) \right|
\leq
(1 + 2\kappa)\UncertaintyFprescient + E
\Bigg)
\geq 1 - \delta.
\]
\end{lemma}

\subsection{Algorithmic Guarantees}\label{subsec:cbai-quantile}
We now apply these statistical results for MAD estimation to show that the \CBAIspace algorithms presented earlier in Subsection~\ref{subsec:cbai-pac} -- with no additional modifications -- can also provide quality guarantees for the arm they select. That is, we show that in addition to outputting the best (or an approximately best) arm, these algorithms also provide quality guarantees to a certain precision ``for free'' without needing extra samples.
\par For simplicity, we will focus on error guarantees for Algorithm~\ref{alg:Simple}. A nearly identical (yet technically hairier) argument yields a similar guarantee for our adaptation of the Successive Elimination algorithm for \CBAIspace (Algorithm~\ref{alg:SuccElimCBAI}).
\par We will make the assumption that the arms distributions are in $\Fam$ so that the MAD estimation results proved above will apply. The parameters of this family are assumed to be known to the algorithm beforehand. Distributions in $\Fam$ have the following property that their lower tails are controlled by their median and MAD.
\begin{lemma}\label{lem:fam-guarantees}
Let $F \in \Fam$ and $Y \sim F$. Then simultaneously for all $t \in [0, \tmax]$,
\[
\Prob\Big(
Y \geq m_1(F) - tBm_2(F)
\Big)
\geq \half + t.
\]
\end{lemma}
\begin{proof}
By item (6) of Lemma~\ref{lem:fam-properties}, we have that for all $t \in [0, \tmax]$, $t = \half - (\half - t) \geq \tfrac{1}{Bm_2(F)} (F^{-1}(\half) - F^{-1}(\half - t))$. Rearranging yields $F^{-1}(\half - t) \geq m_1(F) - Bm_2(F)t$ and completes the proof.
\end{proof}

With this lemma in hand, we turn to guarantees for the arm returned by Algorithm~\ref{alg:Simple}. The following guarantees hold for all adversarial settings since the adversarial strength is encapsulated in the definition of the unavoidable uncertainty terms $U_i$.

\begin{theorem}\label{thm:guarantee-selected}
Let $F_1,\dots,F_k \in \Fam$. Let $\eps \in (0, \min(\frac{2\tmax}{1+2\tmax}, \tfrac{1}{B}))$ and $\bar{U} := U_{\eps, B, \bar{m}_2}$ if the adversary is oblivious or prescient, or let $\eps \in (0, \min(\tmax, \tfrac{1}{B}))$ and $\bar{U} := U_{\eps,B,\bar{m}_2}^{\textsc{(malicious)}}$ if the adversary is malicious.
Consider using Algorithm~\ref{alg:Simple} for \CBAIspace exactly as in Theorem~\ref{Thm:Simple}. Let $\hat{I}$ denote the arm it outputs, let $Y$ be a random variable whose conditional distribution on $\hat{I}$ is $F_{\hat{I}}$, and let $\hat{m_1}$ and $\hat{m_2}$ denote the empirical median and MAD, respectively, from the samples it has seen from $F_{\hat{I}}$. Then, simultaneously for all $t \in [0, \tmax]$,
\[
\Prob \Big(
Y \geq
\left(\hat{m}_1 - tB\hat{m}_2\right) - \left((\half + 2\kappa tB) \alpha + (1 + (1+2\kappa)Bt) \bar{U} \right)
\Big)
\geq
\half + t - \tfrac{3 \delta}{k}.
\]
\end{theorem}
\begin{proof}
By definition of Algorithm~\ref{alg:Simple}, $\hat{I}$ is sampled $n_{\alpha/2,\delta/k} =2 \max(\tfrac{4B^2\bar m_2^2}{\alpha^2}, \,  (\tmax - \tfrac{\eps}{2(1-\eps)})^{-2} ) \log \tfrac{2k}{\delta}$ times against oblivious and prescient adversaries, and $n_{\alpha/2,\delta/k} =2 \max\left(\tfrac{4B^2\bar m_2^2}{\alpha^2}, \,  \left(\tmax - \eps\right)^{-2} \right)  \log\tfrac{3k}{\delta}$ times against malicious adversaries. Therefore -- by Corollary~\ref{corol:est-med} and Lemma~\ref{lem:est-mad} for oblivious and prescient adversaries, or similarly by Corollary~\ref{corol:est-med-malicious} and Lemma~\ref{lem:est-mad-malicious} for malicious adversaries -- we have by a union bound that with probability at least $1 - \tfrac{3\delta}{k}$, both of the following inequalities hold:
\begin{align*}
m_1(F_{\hat{I}}) &\geq \hat{m}_1 - \bar{U} - \tfrac{\alpha}{2}, \; \text{ and}\\
m_2(F_{\hat{I}}) &\leq \hat{m}_2 + (1 + 2\kappa) \bar{U} + 2\kappa \alpha.
\end{align*}
Whenever this event occurs, we have
\[
m_1(F_{\hat{I}}) - tBm_2(F_{\hat{I}}) \geq \left(\hat{m}_1 - tB\hat{m}_2\right) - \left((\half + 2\kappa t B)\alpha + (1 + (1+2\kappa)Bt) \bar{U} \right).
\]
Applying Lemma~\ref{lem:fam-guarantees} and taking a union bound completes the proof.
\end{proof}
Note that the bound inside the probability term in Theorem~\ref{thm:guarantee-selected} can be far less conservative than the crude bound $\hat{m}_1 - tB\bar{m}_2$ if $\eps$ and $\alpha$ are small.

%% file: 6_conclusion.tex


\section{Conclusion}\label{sec:conclusion}
In this paper, we proposed the Best Arm Identification problem for contaminated bandits (\CBAI). This setup can model many practical applications that cannot be modeled by the classical bandit setup. On the way to efficient algorithms for \CBAI, we developed tight, non-asymptotic sample-complexity bounds for estimation of the first two robust moments (median and median absolute deviation) from contaminated samples. These results may be of independent interest, perhaps as ingredients for adapting other online learning techniques to similar contaminated settings.

We formalized the contaminated bandit setup as a special case of the more general partially identifiable bandit problem (\PIBAI), and presented ways to adapt celebrated Best Arm Identification algorithms for the classical bandit setting to this \PIBAIspace problem. The sample complexity is essentially changed only by replacing the suboptimality ``gaps'' with suboptimality ``effective gaps'' to adjust for the challenge of partial identifiability. We then showed how these algorithms apply to the special case of {\CBAIspace} by making use of the aforementioned guarantees for estimating the median from contaminated samples. We complemented these results with nearly matching information-theoretic lower bounds on the sample complexity of \CBAI, showing that (up to a small logarithmic factor) our algorithms are optimal. This answers the question this paper set out to solve: determining the complexity of finding the arm with highest median given contaminated samples. Finally, we used our statistical results on estimation of the second robust moment from contaminated samples to show that without additional samples, our algorithms can also output quality guarantees on the selected arm.

Our paper suggests several potential directions for future work.
\begin{itemize}
\item Our upper and lower bounds on the sample complexity of \CBAIspace are off by a logarithmic factor in the number of arms. It is an interesting open question to close this gap, especially since many of the tricks for doing this in the classical \BAIspace setup do not seem to extend to this partially identifiable setting (see Remark~\ref{rem:adapt}). One concrete possibility is to obtain a tighter upper bound by adapting the algorithm of~\citep{jamieson2014lil} to the \CBAIspace problem. This would require proving a non-asymptotic version of the Law of the Iterated Logarithm for empirical medians, which would be of independent interest.
\item We have developed algorithms for online learning in the presence of partial identifiability. How far does this toolkit extend? In particular, do our results apply to more complicated or more general feedback structures such as partial monitoring (see e.g.~\citep{bartok2014partial}) or graph feedback (see e.g.~\citep{alon2017nonstochastic})?
\item The contaminated bandit setup models many real-world problems that cannot be modeled by the classical bandit setup. Is it applicable and approachable to formulate other classical online-learning problems in similar contamination setups?
\item More abstractly, we think that problems at the intersection of online learning and robust statistics are not only mathematically rich, but also are increasingly relevant, given the recent influx of active learning tasks with data that are not completely trustworthy. It may be valuable to use techniques from one of the fields to approach problems in the other, as we did here.
\end{itemize}

%% file: app_robestdetails.tex

\section{Properties of $\Famnosig$ and $\Fam$}\label{app:prop-fam}
The following lemma lists some simple properties of $\Famnosig$ (defined in Definition~\ref{def:2-param}) and $\Fam$ (defined in Definition~\ref{def:fam-mad}), which we use often throughout the paper.
\par For shorthand, we denote $I_{F,\tmax} :=  \left[\QLF(\half - \tmax), \QRF(\half + \tmax)\right]$ for a distribution $F$ and a real number $\tmax \in (0, \half)$. 

\begin{lemma}\label{lem:fam-properties}
If $F \in \Famnosig$, then:
\begin{enumerate}
\item For any $a\neq 0$ and $b\in\Reals$, the distribution $F(a\cdot + b)$ is also in $\calF_{B,\tmax}$.
\item $F$ is strictly monotonically increasing in $I_{F, \tmax}$.
\item $\QLF(t)=\QRF(t)$, for all $t\in F(I_{F,\tmax}) = [\half \plusminus \tmax]$.
\item $\QLF(F(x)) = x$ for all $x \in I_{F, \tmax}$ and we write $F^{-1}$ for $\QLF = \QRF$.
\item The left and right quantiles of $F$ are equal in the interval $F(I_{F, \tmax}) = [\half -\tmax, \half + \tmax]$.
\item For any $u_1, u_2 \in F(I_{F, \tmax}) =[\half -\tmax, \half + \tmax]$,
\[
|u_1 - u_2| \geq \frac{1}{B\sigmed(F)} |F^{-1}(u_1) - F^{-1}(u_2)|.
\]
\end{enumerate}
Moreover if we also have $F \in \Fam$, then:
\begin{enumerate}
\item[7.] For any $a \neq 0$ and $b\in\Reals$, the distribution $F(a\cdot + b)$ is in $\calF_{B, \tmax, a\sigmedmax, \kappa}$.
\item[8.] Let $H$ be the distribution of $|Y-m|$, where $Y \sim F$. Then $H \in \mathcal{F}_{\bar{t}_H, B_H}$, where
$B_H := \kappa B$ and $\bar{t}_H := \min(\half, \tfrac{2}{B})$.
\end{enumerate}
\end{lemma}
\begin{proof}
When proved in order, all of these statements follow easily from the definition of the function class and the earlier statements. The only part requiring effort is item 8, which we now verify. Fix any $r_1, r_2 \in I_{H,\bar{t}_H}$, where without loss of generality $r_1 \geq r_2 \geq 0$. Then
\begin{align}
H(r_1) - H(r_2)
&= \Prob(|Y - m_1(F)| \leq r_1) - \Prob(|Y - m_1(F)| \leq r_2) \nonumber
\\ &=
\left[\Prob(Y \leq m_1(F) + r_1) - \Prob(Y \leq m_1(F) + r_2)\right] \nonumber
\\ &\;\;\;+ \left[
\Prob(Y < m_1(F) - r_2) - \Prob(Y < m_1(F) - r_1)
\right] \nonumber
\\ &\geq F(m_1(F) + r_1) - F(m_1(F) + r_2) \nonumber
\\ &\geq \frac{|r_1 - r_2|}{Bm_2(F)} \label{eq:proof:fam-properties:range}
\\ &\geq \frac{|r_1 - r_2|}{\kappa Bm_4(F)} \nonumber
\\ &= \frac{|r_1 - r_2|}{B_H m_2(H)}. \nonumber
\end{align}
The only step requiring justification is the inequality in \eqref{eq:proof:fam-properties:range}: this is evident by \eqref{eq:def-Fam} if $r_1, r_2 \in [m_1(F) \plusminus 2m_2(F)]$, but this condition must be checked. Once we show this condition is met, we are immediately done.
\par Therefore, it is now sufficient to prove that $r_1, r_2 \in [m_1(F) \plusminus 2m_2(F)]$. Since $r_1, r_2 \geq 0$, it suffices to show that the largest value in $I_{H,\bar{t}_H}$, namely $Q_{R,H}(\half + \bar{t}_H)$, is at most $2m_2(F)$. And to show this, it suffices to show $\half + \bar{t}_H < H(2m_{2}(F))$. We show this last inequality presently: by an similar argument as in the first few lines of the above display,
\begin{align}
H(2m_{2}(F)) - \half
&= H(2m_{2}(F)) - H(m_{2}(F)) \nonumber
\\ &= \left[\Prob(Y \leq m_1(F) + 2m_2(F)) - \Prob(Y \leq m_1(F) + m_2(F))\right] \nonumber
\\ &\;\;\;+ \Prob(Y \in [m_1(F) - 2m_2(F), m_1(F) - m_2(F))) \nonumber
\\ &>F(m_1(F) + 2m_{2}(F)) - F(m_1(F) + m_{2}(F))
\label{eq:proof:fam-properties:range-2}
\\ &\geq \frac{1}{B}\label{eq:proof:fam-properties:range-3}
\\ &\geq \bar{t}_H, \nonumber
\end{align}
where \eqref{eq:proof:fam-properties:range-2} is because $F$ is monotonically increasing in $[m_1(F) - 2m_2(F), m_1(F)+2m_2(F)]$ by property (i) in the definition of $\Fam$, and \eqref{eq:proof:fam-properties:range-3} is due to \eqref{eq:def-Fam}. This completes the proof.
\end{proof}

%% file: app_estimation.tex

\section{Proofs for Estimation from Contaminated Samples}\label{app:est}
Throughout, we denote the indicator random variable for an event $\mathcal E$ by $\mathds{1}\{\mathcal E\}$.

\subsection{Estimation of Median for Oblivious and Prescient Adversaries}\label{app:est:median}
\begin{proof}[Proof of Lemma~\ref{lem:est-med-R}]
For shorthand, let $\hat{m} := \hatm(X_1, \dots, X_n)$ and let $m \in m_1(F)$ be any median of $F$.
For each $i \in [n]$, define the indicator random variable
\[
L_i := \mathds{1}\left\{(D_i = 1) \text{ or } \left(D_i = 0 \text{ and } Y_i \geq \QRF\left(\tfrac{1}{2(1-\eps)}+a\right) \right) \right\},
\]
where $a := \tfrac{\sqrt{\log (2/\delta)}}{(1-\eps)\sqrt{2n}} <  \sqrt{\tfrac{2\log (2/\delta)}{n}}$. By independence of $D_i$ and $Y_i$, $L_i$ has mean 
\[
\E\left[L_i\right] =
\eps + (1-\eps)\left(1 - \left(\tfrac{1}{2(1-\eps)}+a\right)\right) 
= \tfrac{1}{2} - (1 - \eps)a.
\]
Moreover, the $\{L_i\}_{i \in [n]}$ are independent, and thus by Hoeffding's inequality,
\begin{align*}
\Prob\left(\hat{m} \geq \QRF\left(\frac{1}{2(1-\eps)} + a \right) \right)
\leq \Prob\left(\sum_{i=1}^n L_i \geq \frac{n}{2} \right)
\leq \exp\left(-2n(1-\eps)^2a^2\right)
= \frac{\delta}{2}.
\end{align*}
Therefore, with probability at least $1 - \tfrac{\delta}{2}$,
\[
\hat{m}- m
<
\QRF\left(\frac{1}{2(1-\eps)} + a \right) -\QRF\left(\frac{1}{2}\right)
\leq R\left(
\frac{\eps}{2(1-\eps)} + a 
\right),
\]
where the final inequality is due to \eqref{eq:def-R}, which we may invoke since $\half + \tfrac{\eps}{2(1-\eps)} + a \leq \half + \tmax$ by our choice of $n$. An identical argument (or by symmetry with $-F$) yields the analogous result for the lower tail of $\hat{m}$, namely that $m - \hat{m} <  R(\frac{\eps}{2(1-\eps)} + a)$ with probability at least $1 - \tfrac{\delta}{2}$. The lemma statement follows by a union bound.
\end{proof}

\subsection{Estimation of Median for Malicious Adversaries}
We now turn to providing the proofs of Lemma~\ref{lem:est-med-malicious-R} and~\ref{lem:est-med-malicious-R-tight}. Below, it will be convenient to denote by $y_{(k)}$ the $k$-th order statistic of a (possibly random) sequence $y_1, \dots, y_n \in \Real$. 
\par Informally, the proof of Lemma~\ref{lem:est-med-malicious-R} proceeds by (i) showing that $\hat{m}_1(X_1, \dots, X_n)$ is deterministically bounded within the order statistics $Y_{(\lfloor \tfrac{n}{2}\rfloor \plusminus \sum_{i=1}^n D_i)}$ (Lemma~\ref{lem:med-contaminate-k}), (ii) applying Hoeffding's inequality to argue that w.h.p., at most $\sum_{i=1}^n D_i \approx \eps n$ samples are contaminated, and (iii) reusing the techniques of Lemma~\ref{lem:est-med-R} to argue that w.h.p., the order statistics of $Y_{(\lfloor \tfrac{n}{2}\rfloor \plusminus \eps n)}$ are within the desired error range from $\med(F)$. Since each of these steps is tight up to a small amount of slack, the proof of the converse Lemma~\ref{lem:est-med-malicious-R-tight} proceeds essentially by just showing that each of these steps occurs also in the opposite direction w.h.p.
\par The following lemma will be helpful for step (i). Its proof is straightforward by induction on $s$ and is thus omitted.

\begin{lemma}\label{lem:med-contaminate-k}
Let $x_i := d_i y_i + (1 - d_i) z_i$, where $y_1 \leq \dots \leq y_n$ and $z_1, \dots, z_n$ are arbitrary real-valued sequences, and $d_1, \dots, d_n$ is an arbitrary binary-valued sequence satisfying $s := \sum_{i=1}^n d_i < \tfrac{n}{2}$. Then
\[
y_{\left(\left\lfloor\frac{n}{2}\right\rfloor  - s\right)}
\leq 
\hatmed(x_1, \dots, x_n)
\leq 
y_{\left(\left\lceil\frac{n}{2}\right\rceil+ s\right)}.
\]
\end{lemma}

\begin{proof}[Proof of Lemma~\ref{lem:est-med-malicious-R}]
Define for shorthand $a := \sqrt{\tfrac{\log(3/\delta)}{2n}}$. By Hoeffding's inequality, the event $E := \{\sum_{i=1}^n D_i \leq (\eps + a) n \}$ occurs with probability at least $\Prob(E) \geq 1 - \tfrac{\delta}{3}$. Since $\eps + a \leq \tmax < \half$ by our choice of $n$,  Lemma~\ref{lem:med-contaminate-k} implies that, whenever $E$ occurs,
\[
Y_{\left(\left\lfloor\frac{n}{2}\right\rfloor - \lfloor (\eps + a) n \rfloor\right)}
\leq
\hatmed(X_1, \dots, X_n)
\leq
Y_{\left(\left\lceil\frac{n}{2}\right\rceil + \lfloor (\eps + a ) n \rfloor\right)}
\]
also occurs.
Now, define for each $i \in [n]$ the indicator random variable $L_i := \mathds{1}\left\{Y_i > \QRF(\half + \eps + 2a) \right\}$. Then $\E[L_i] \leq \half - \eps - 2a$, so by Hoeffding's inequality,
\begin{align*}
\Prob\Bigg(Y_{\left(\left\lceil\frac{n}{2}\right\rceil + \lfloor (\eps + a ) n \rfloor\right)}>\QRF(\half + \eps + 2a)\Bigg)
\leq \Prob\left(\sum_{i=1}^n L_i \geq (\half - \eps - a)n \right)
\leq \exp(-2na^2)
= \frac{\delta}{3}.
\end{align*}
An identical argument (or simply by symmetry on $F(-\cdot)$) also yields that $\DS Y_{\left(\left\lfloor\frac{n}{2}\right\rfloor - \lfloor (\eps + a) n \rfloor\right)} < \QLF(\half - \eps - 2a)$ with probability at most $\tfrac{\delta}{3}$. We conclude by a union bound that with probability at least $1 - \delta$,
\[
\QLF\left(\half - (\eps + 2a)\right)
\leq
\hatmed(X_1, \dots, X_n)
\leq 
\QRF\left(\half + (\eps + 2a)\right).
\]
Whenever this occurs, we have by virtue of \eqref{eq:def-R} that $\sup_{m \in \med(F)} \left| \hatmed(X_1, \dots, X_n) - m \right|
\leq R\left(\eps + 2a \right)$,
since both $\half \plusminus (\eps + 2a) \in [\half \plusminus \tmax]$ by our choice of $n$.
\end{proof}

\begin{proof}[Proof of Lemma~\ref{lem:est-med-malicious-R-tight}]
  Consider any distribution $F$ with unique median $0$ satisfying $R(t) = \QRF(\half + t) = -\QLF(\half - t)$ for each $t \in \tmax$. Such an $F$ can be constructed by starting with a Dirac measure $\delta_0$ at zero and then pushing mass to the tails as far as~\eqref{eq:def-R} allows. Next, consider the joint distribution on $(D,Y,Z)$ where $Y \sim F$, the conditional distribution of $D$ given $Y$ is $\Ber(2\eps \cdot \mathds{1}\{Y \leq 0\})$ and $Z \sim \delta_{\QRF(\half + \eps)}$. The marginal of $D$ is easily seen to be correct, since 
\[
\Prob(D = 1) = \Prob(D = 1 | Y \leq 0) \cdot \Prob(Y \leq 0) + \Prob(D = 1|Y > 0)\cdot \Prob(Y > 0) = 2\eps \cdot \half = \eps.
\]
Let $a := \sqrt{\tfrac{\log (1/\delta)}{2n}}$, and note that $a < \eps$ by our choice of $n$. For each $i \in [n]$, define the indicator random variable
\[
L_i := \mathds{1}\left\{(Y_i > R(\eps - a) ) \text{ or } (Y_i \leq 0 \text{ and } D_i = 1) \right\}.
\]
Each of these has mean
\[
\E[L_i] \leq 1 - (\half + (\eps - a)) + \half (2\eps) = \half + a.
\]
Therefore, by Hoeffding's inequality, we conclude that
\begin{align*}
\Prob\Big(
\left| \hatmed(X_1, \dots, X_n) - \med(F) \right| <
R\left(\eps  - a \right)
\Big)
& \leq \Prob\Big(\hatmed(X_1, \dots, X_n) < \med(F) + R(\eps - a) \Big)
\\ &\leq \Prob\left(
\sum_{i=1}^n L_i \geq \frac{n}{2}
\right)
\leq \exp\left(
-2a^2n
\right)
= \delta.
\end{align*}
\end{proof}

\subsection{Estimation of Second Robust Moment}\label{app:est:mad}

Here, we prove Lemma~\ref{lem:est-mad}. This is done by decomposing the MAD estimation error into two terms, each of which resembles the error between a true median (of a distribution related to $F$) and an empirical median of contaminated samples, and then applying the median estimation guarantees from the previous section. We begin by proving several helpful lemmas. 

\begin{lemma}\label{lem:med-lipschitz-helper}
For any (possibly random) sequence $x_1, \dots, x_n \in \Real$ and any $c \in \Real$,
\[
\Big|\hatmed(|x_1 + c|, \dots, |x_n + c|) - \hatmed(|x_1|, \dots, |x_n|)\Big| \leq |c|
\]
\end{lemma}
holds almost surely, over the possible randomness of the sequence.
\begin{proof}
For any fixed vector $x \in \Real^n$, the function $c \mapsto \hatmed(|x_1 + c|, \dots, |x_n + c|)$ is $1$-Lipschitz since it is the composition of the following $1$-Lipschitz functions with respect to the $L_{\infty}$ norm: adding $c\vec{1}$, taking entrywise absolute values, and taking an order statistic. 
\end{proof}

We also give the following distributional version of Lemma~\ref{lem:med-lipschitz-helper}.

\begin{lemma}\label{lem:med-lipschitz-helper-Popul}
Consider any real-valued random variable $X$ and any $c \in \Real$. Assume that $m_1(|X|)$ and $m_1(|X+c|)$ are unique. Then
\[
\big| \med(|X+c|) - \med(|X|) \big| \leq |c|.
\]
\end{lemma}
\begin{proof}
Assume without loss of generality that $c\geq 0$; the case $c < 0$ follows by an identical argument or by symmetry. For shorthand, let $m := \med(|X|)$ and let $\tilde{m} := \med(|X+c|)$. First, we show that $\tilde{m} \leq m+c$, i.e., we show that $\PP(|X+c|\leq m+c)\geq \half$. This is straightforward: $\PP(|X+c|\leq m+c)=\PP(-m-2c\leq X\leq m) \geq \PP(-m\leq X\leq m)=\PP(|X|\leq m)\geq \half$, where we the first step is by non-negativity of $m$ and the last step is by definition of $m$. A nearly identical argument shows $m \leq \tilde{m} + c$, namely $\Prob(|X| \leq \tilde{m} + c) = \Prob(-\tilde{m} - c \leq X \leq \tilde{m} + c) = \Prob(-\tilde{m} \leq X \leq \tilde{m}+ 2c) \geq \Prob(-\tilde{m} \leq X \leq \tilde{m}) = \Prob(|X| \leq \tilde{m}) \geq \half$. We therefore conclude that $|\tilde{m} - m| \leq c$.
\end{proof}


\begin{lemma}\label{lem:bound-2-4}
For any distribution $F$ with unique $m_2(F)$ and $m_4(F)$,
\[
m_4(F) \leq 2m_2(F).
\]
\end{lemma}
\begin{proof}
Let $Y \sim F$. By Lemma~\ref{lem:med-lipschitz-helper-Popul},
\begin{align*}
m_4(F)
= m_1\left(
\big|\left|
Y - m_1(F)
\right|
- m_2(F)\big|
\right)
\leq m_1\left(\left|Y - m_1(F) \right| \right) + m_2(F)
= 2m_2(F).
\end{align*}
\end{proof}

We are now ready to prove Lemma~\ref{lem:est-mad}.

\begin{proof}[Proof of Lemma~\ref{lem:est-mad}]
  For shorthand, denote $\hat{m} := \hatmed(X_1, \dots, X_n)$, and let $H$ denote the distribution of $|Y - m_1(F)|$ where $Y \sim F$. Combining the fact that $m_2(F) = m_1(H)$ with Lemma~\ref{lem:med-lipschitz-helper} and the triangle inequality, the MAD estimation error is bounded above by
\begin{align}
&\;\Big|\hatsigmed(X_1, \dots, X_n) - \sigmed(F)\Big| \nonumber
\\ =&\;\Big| \hatmed\left(|X_1 - \hat{m}|, \dots, |X_n - \hat{m}|\right) - m_1(H) \Big| \nonumber
\\ \leq&\;
\Big| \hatmed\left(|X_1 - \med(F)|, \dots, |X_n - \med(F)|\right)- m_1(H) \Big|  + \Big|\hat{m} - \med(F)\Big|.
\label{eq:lem-est-mad:decompose}
\end{align}
By Corollary~\ref{corol:est-med} and our choice of $n$, the second error term in \eqref{eq:lem-est-mad:decompose} is bounded above by $\UncertaintyF + \tfrac{E}{2}$ with probability at least $1 - \tfrac{\delta}{2}$. To control the first error term in \eqref{eq:lem-est-mad:decompose}, we apply Corollary~\ref{corol:est-med} with the distribution $H$ in lieu of $F$ and  the contaminations $\tilde{Z}_i := |Z_i - \med(F)|$ in lieu of $Z_i$. Thus, by combining item 8 in Lemma~\ref{lem:fam-properties} with Corollary~\ref{corol:est-med}, this first error term is bounded above by $U_{\eps, \kappa B,m_2(H)} + \tfrac{E}{2}$ with probability at least $1 - \tfrac{\delta}{2}$ whenever we have at least $ 2\max\left(
4B^2\kappa^2m_2^2(H)E^{-2}, (\min(\half, \tfrac{1}{B}) - \tfrac{\eps}{2(1-\eps)})^{-2}
\right) \log\tfrac{4}{\delta}$ samples, which is satisfied because of our choice of $n$ and the inequality $m_2(H) = m_4(F) \leq 2m_2(F)$ from Lemma~\ref{lem:bound-2-4}. Therefore, a union bound implies that, with probability at least $1 - \delta$, the MAD estimation error is at most $\UncertaintyF + U_{\eps, \kappa B,m_2(H)} + E$. By another application of the inequality $m_2(H) \leq 2m_2(F)$, this is bounded above by $(1 + 2\kappa)\UncertaintyF + E$, as desired.
\end{proof}

\par The proof of Lemma~\ref{lem:est-mad-malicious} -- the analogous result to Lemma~\ref{lem:est-mad} but for the \emph{malicious} adversarial setting -- is omitted since it is identical to the proof of Lemma~\ref{lem:est-mad} with the uses of Corollary~\ref{corol:est-med} replaced by uses of Corollary~\ref{corol:est-med-malicious}.

%% file: app_algdetails.tex

\section{Proofs for Adaptation of the Successive Elimination Algorithm}\label{app:algorithms}

We first define some notation. Let $c$ be a constant such that $n_{\alpha, \delta} \leq \frac{c}{\alpha^2}\logdel$ for all $\alpha > 0$ and $\delta \in (0, 1)$; such a constant clearly exists by definition of \PIBAIspace (see Equation~\ref{eq:estimator}). Let $R$ denote the number of total rounds in Algorithm~\ref{alg:SuccElim} before termination, and for each round $r \in [R]$, let $S_r$ denote the set of all arms still in $S$ when entering round $r$. For succintness, we also denote $\delta_r := \tfrac{6\delta }{\pi^2kr^2}$.

\begin{proof}[Proof of Theorem~\ref{Thm:SuccElim}]
Without loss of generality, assume that the best arm is $i^* = 1$. Define the event $E := \{|\hat p_{i,r}-p_i|\leq U_i+\alpha_{r,\delta_r}, \,\forall r\geq 1,\, \forall i\in S_r\}$. Recall that $\hat p_{i,r}$ is the estimate of the measure of quality for arm $i$ after $r$ samples, and that arm $i$ will stop being pulled once $i\notin S_r$. For analysis purposes, consider 
consider virtual estimates $\hat p_{i,r}$ that would be obtained if we continued to pull the eliminated arms indefinitely, so that $\hat p_{i,r}$ is defined for all $i \in [k]$ and $r \in \mathbb{N}$. By a union bound,
\begin{align*}
	\PP\left(E^C\right) 
	\leq \sum_{i=1}^k\sum_{r=1}^\infty \PP\left(|\hat p_{i,r}-p_i|> U_i+\alpha_{r,\delta_r}\right)
	\leq \sum_{i=1}^k\sum_{r=1}^\infty \delta_r
	= \sum_{i=1}^k \sum_{r=1}^{\infty} \frac{6\delta }{\pi^2kr^2}
	= \delta,
        \end{align*}      
where above we have used \eqref{eq:estimator} and the famous Basel identity.
\par We conclude from the above that $E$ occurs with probability at least $1 - \delta$. Henceforth let us assume that $E$ occurs. A simple induction argument shows that $1\in S_r$ for each round $r\in [R]$, which implies that Algorithm~\ref{alg:SuccElim} returns the optimal arm. Indeed, at each round $r$ for which $1\in S_r$, $E$ guarantees that for all $j\in S_r$,
\[
\hat p_{1,r}
\geq p_1-U_1-\alpha_{r,\delta_r}
= \tilde\Delta_j + p_j + U_j - \alpha_{r,\delta_r}
> p_j + U_j - \alpha_{r,\delta_r}
\geq \hat{p}_{j,r} - 2\alpha_{r,\delta_r}.
\]
and so by definition of Algorithm~\ref{alg:SuccElim}, the optimal arm $1$ is not eliminated at round $r$.
\par Now, still assuming that $E$ occurs, let us bound the sample complexity $T$. For each $i \in [k]$, denote by $T_i$ the number of times that arm $i$ is pulled. Clearly, $T=\sum_{i=1}^k T_i$. Moreover, since arm $1$ is never eliminated so long as $E$ occurs, $T\leq 2\sum_{i=2}^k T_i$. For each $i\geq 2$, arm $i$ is eliminated no later than the first round $r$ in which $\hat p_{i,r}<\hat p_{1,r}-2\alpha_{r,\delta_r}$ which, by $E$, is satisfied as soon as $p_i+U_i+\alpha_{r,\delta_r}<p_1-U_1-3\alpha_{r,\delta_r}$, or equivalently $\tilde\Delta_i>4\alpha_{r,\delta_r}$. We conclude that arm $i$ is eliminated in the first round $r$ where
\[
\tilde\Delta_i > 4\alpha_{r,\delta_r} = 4\sqrt{\frac{c \log(\tfrac{\pi^2kr^2}{6\delta})}{r}},
\]
which occurs for $r\leq C\frac{1}{\tilde\Delta_i^2} \log \left(\tfrac{k}{\delta \tilde\Delta_i}\right)$ for some universal constant $C>0$. Hence, 
\begin{equation*}
	T \leq \sum_{j=2}^k T_j = O\left(\sum_{j=2}^k \frac{1}{\tilde\Delta_i^2} \log \left(\frac{k}{\delta \tilde\Delta_i}\right)  \right).
\end{equation*}
\end{proof}

\begin{proof}[Proof of Theorem~\ref{Thm:SuccElimCBAI}]
The proof is nearly identical to that of Theorem \ref{Thm:SuccElim}. The main difference in the proof of correctness is that at each round $r\geq 0$, if any arm is still in $S_r$, it has been pulled at least
\begin{align*}
N\log\left(\frac{\pi^2 k}{2\delta} \right)
+
2N \sum_{t=1}^r \log \left(\frac{t+1}{t} \right)
+
r
&=
N\log\left(\frac{\pi^2 k(r+1)^2}{2\delta} \right)
+ r
\\ &\geq 
\left(\frac{2B^2\bar m_2^2}{\alpha_{r,\delta_r}^2} + N \right)\log\frac{3}{\delta_r}
\\ &\geq 
\max\left(\frac{2B^2\bar m_2^2}{\alpha_{r,\delta_r}^2}, N \right)\log\frac{3}{\delta_r}
\end{align*}
times. Thus, we may apply Corollary~\ref{corol:est-med} (for oblivious or prescient adversaries) or Corollary~\ref{corol:est-med-malicious} (for malicious adversaries) to obtain the identical guarantees as needed in the proof of Theorem~\ref{Thm:SuccElim} for estimating the median up to accuracy $\alpha_{r,\delta_r}$ with probability at least $1 - \delta_r$. Hence, the $(0,\delta)$-PAC guarantee follows easily by identical reasoning. The only difference in the sample complexity proof is that it we must keep track of the additional draws
\[
\sum_{i=1}^k \left[\left \lceil N \log \left(\frac{\pi^2 k}{2\delta}\right) \right \rceil + \sum_{t=1}^{T_i} \left \lceil 2N \log\left(\frac{t+1}{t}\right) \right \rceil \right]
\leq 
\sum_{i=1}^k \left[
(T_i + 1)
+
N\log \left(\frac{\pi^2 k (T_i+1)^2}{2\delta} \right) 
\right]
,
\]
where $T_i$ is the number of times that arms $i$ is pulled. By what was shown in Theorem~\ref{Thm:SuccElim}, the first term $\sum_{i=1}^k (T_i + 1) = O(\sum_{i \neq i^*} \frac{1}{\tilde\Delta_i^2} \log (\frac{k}{\delta \tilde\Delta_i}) )$, and the second term $\sum_{i=1}^k N\log (\frac{\pi^2 k (T_i+1)^2}{2\delta} )  = 
O(N\sum_{i \neq i^*}\log(\frac{k}{\delta \tilde{\Delta}_i}\log(\frac{k}{\delta \tilde{\Delta}_i}) )) = O(N\sum_{i \neq i^*}\log(\frac{k}{\delta \tilde{\Delta}_i}))$.
\end{proof}

%% file: app_lb.tex

\section{Proofs for Lower Bounds}\label{app:lb}
In this section, we make the proof sketch in Subsection~\ref{subsec:lb} formal. The proof is broken into two parts. First, we exhibit hard instances for \BAIspace in which the arms all have smoothed Bernoulli distributions. Second, we reduce this instance into a lower bound instance for \CBAI.
\par Throughout, we adopt the standard assumption in the multi-armed bandit literature \citep{mannor2004sample} that all algorithms we consider have a stopping time $T$ which is almost surely finite.

 
\subsection{\BAIspace Lower Bound}
Here, we prove a gap-dependent lower bound for the \BAIspace problem that we will make use of in the following subsection. \cite{mannor2004sample} gave the first such lower bound by exhibiting hard instances for \BAIspace using Bernoulli-distributed arms. However, our \CBAIspace reduction will not work with Bernoulli distributions (see the proof sketch in Subsection~\ref{subsec:lb}), and so here we exhibit hard instances for \BAIspace using instead a distribution that will work for our \CBAIspace reduction, namely the smoothed Bernoulli distribution.
\par Recall that we define the smoothed Bernoulli distribution with parameter $p\in (0,1)$, denoted by $\SBer(p)$, to be the uniform mixture of the Bernoulli distribution with parameter $p$ and the uniform distribution on $[0,1]$ with weights $\half$ and $\half$. It is easy to see that for $p,q\in (0,1)$, the Kullback-Leibler divergence between $\SBer(p)$ and $\SBer(q)$ is given by 
\begin{align*}
\textsf{KL}(\SBer(p),\SBer(q)) & = \frac{1}{2}\textsf{KL}(\Ber(p),\Ber(q)) \\
& = \frac{1}{2}\left(p\log\left(\frac{p}{q}\right)+(1-p)\log\left(\frac{1-p}{1-q}\right)\right).
\end{align*}
In particular, a second-order Taylor expansion yields that, for all $\eta\in (0,\half)$, there exists some positive constant $C_\eta$ such that the inequality
\begin{equation} \label{BoundKLSBer}
	\textsf{KL}(\SBer(p),\SBer(q))\leq C_\eta(p-q)^2
\end{equation}
holds for all $p,q\in [\eta,1-\eta]$.
\par With~\eqref{BoundKLSBer} in hand, the following lemma becomes a direct consequence of Lemma 1 and Remark 5 from \cite{kaufmann2016complexity}.

\begin{lemma}\label{lem:smoothed-bernoulli}
Let $k\geq 2$, $\eta\in (0,\half)$, and $p_1,\ldots,p_k\in [\eta, 1-\eta]$. Consider the instance of \BAIspace where arm $i \in [k]$ has distribution $\SBer(p_i)$. Then, for any $\alpha > 0$ and any $\delta \in (0,\tfrac{3}{20})$, any $(\alpha, \delta)$-PAC \BAIspace algorithm must use at least the following number of samples in expectation:
$$\E[T] \geq \frac{C_{\eta}}{4}\left(
\sum_{i \in [k] \setminus \{i*\}} \frac{1}{\max(\Delta_i, \alpha)^2}\log \left(\frac{1}{2.4\delta} \right)
\right),
$$
where $i^* := \argmax_{i \in [k]} p_i$, and $\Delta_i := p_{i^*} - p_i$ for each $i \in [k] \setminus \{i^*\}$.

\end{lemma}

\subsection{CBAI Lower Bound}\label{app:lb-2}
In this subsection, we show how to prove the lower bounds for \CBAIspace using the technique sketched in Subsection~\ref{subsec:lb}. In particular, we will show how to prove Theorem~\ref{thm:lb-pac} by ``\CBAI-lifting'' the MAB instances we proved were hard in Lemma~\ref{lem:smoothed-bernoulli}. The proof of Corollary~\ref{corol:lb} then follows immediately by letting $\alpha \to 0$; since an algorithm that returns the \textit{best} arm with probability at least $1 - \delta$, is by definition a $(0, \delta)$-PAC algorithm.

\begin{proof}[Proof of Theorem~\ref{thm:lb-pac}]
Fix any $k \geq 2$, $\delta \in (0, \tfrac{3}{20})$, $\alpha \in (0, \tfrac{1}{6})$, $\eps \in (0, \tfrac{1}{15})$, and $\tmax \in (0, \tfrac{1}{10})$.
With $\eta = \tfrac{1}{3}$, Lemma~\ref{lem:smoothed-bernoulli} proves that for any $p_1,\dots,p_k \in [\tfrac 1 3, \tfrac 2 3]$, the \BAIspace instance with $\tilde{F}_i := \SBer(p_i)$ distributed-arms has the property that any ($\alpha$, $\delta$)-PAC \BAIspace algorithm must use at least 
\begin{align}
\Omega\left(\sum_{i \in [k] \setminus \{i^*\}} \frac{1}{\max(\Delta_i, \alpha)^2}\logdel\right)
\label{eq:proof-lb-pac-Omega}
\end{align}
samples in expectation, where $i^* := \argmax_{i \in [k]} p_i$ and $\{\Delta_i\}_{i \in [k] \setminus \{i^*\} }$ are the gaps w.r.t. the arm distributions $\{\SBer(p_i)\}_{i \in [k]}$. Without loss of generality, let us assume $1 = \argmax_{i \in [k]} p_i$ (arm $1$ is the best). At this point we now treat the different adversarial settings separately, since the lower bounds and thus also the liftings differ.

\subsubsection{Lifting for Oblivious and Prescient Adversaries} Define the following distributions:
\begin{align*}
F_1 :&= \frac{1-2\eps}{2(1-\eps)}\Ber(r) + \frac{1}{2(1-\eps)}\UnifI, \quad\text{and}
\\ F_i :&= \frac{1-2\eps}{2(1-\eps)}\Ber(q_i) + \frac{1}{2(1-\eps)}\UnifI \;\;\;\; \forall i \in \{2, \dots, k\},
\end{align*}
where $r := \tfrac{p_1}{1 - 2\eps}$ and $q_i := \tfrac{p_i - 2\eps}{1 - 2\eps}$. It is not hard to see that
\begin{align*}
\tilde{F}_1 := \SBer(p_1) &= (1 - \eps) F_1 + \eps \delta_0, \quad\text{and} \\
\tilde{F_i} := \SBer(p_i) &= (1 - \eps) F_i + \eps \delta_1 \;\;\;\; \forall i \in \{2, \dots, k\}.
\end{align*}
In other words, \textit{samples generated from $\SBer(p_i)$ are equal in distribution to samples generated from the above contaminated mixture model} of $(1 - \eps) F_i$ and $\eps$ times a Dirac measure.
\par Next, a simple calculation shows that $m_1(\tilde{F_1}) = p_1$, $m_1(F_1) = p_1 + \eps$, and $m_2(F_1) = \tfrac{1-\eps}{2}$. Similarly, for any $i \in \{2, \dots, k\}$, we have $m_1(\tilde{F_i}) = p_i$, $m_1(F_i) = p_i - \eps$, and $m_2(F_i) = \tfrac{1-\eps}{2}$. Moreover, for all $i \in [k]$,  we have that  $F_i \in \calF_{B, \tmax}$ for $B = 4$ and any $\tmax < \tfrac{1}{2(1 - \eps)} \min(p_i + \eps, 1 - (p_i + \eps)) \leq \tfrac{1}{2} \cdot \frac{15}{14} \cdot \min(\tfrac{1}{3}, 1 - (\tfrac{2}{3} + \tfrac{1}{15})) = \tfrac{1}{7}$. We conclude that for each $i \in [k]$, the change in median between the distribution $F_i$ and the contaminated distribution $\tilde{F}_i$ is equal to
\[
|m_1(\tilde{F}_i) - m_1(F_i)| = \eps = 
B\sigmed(F_i)\frac{\eps}{2(1-\eps)}
= 
U_{\eps, B,\sigmed(F_i)}.
\]
Therefore, we conclude that running any $(\alpha, \delta)$-approximate \CBAIspace algorithm $\calA$ on the samples obtained from the above \BAIspace instance will result in $\calA$ returning arm $\hat{I}$ satisfying $m_1(F_i) \geq m_1(F_1) - (2U_{\eps, B,\tfrac{1-\eps}{2}} + \alpha)$ with probability at least $1 - \delta$. Whenever this event occurs, we have by the above calculations that $p_{\hat{i}} \geq p_1 - \alpha$. Therefore, $\calA$ returned an $\alpha$ approximate arm with probability at least $1 - \delta$ for this hard \BAIspace instance. We conclude from the lower bound in \eqref{eq:proof-lb-pac-Omega} that $\calA$ must use at least $\Omega\left(
\sum_{i \in [k] \setminus \{i^*\}} \frac{1}{\max(\tilde{\Delta}_i, \alpha)^2}\logdel
\right)$ samples in expectation, where $\{\tilde{\Delta}_i\}_{i \in [k] \setminus \{i^*\} }$ are the effective gaps w.r.t. the arm distributions $\{F_i\}_{i \in [k]}$.

\subsubsection{Lifting for Malicious Adversaries} The idea is similar to the oblivious and prescient case done above. The difference is that malicious adversaries can shift quantiles further (see Corollary~\ref{corol:est-med-malicious}) and so we must exhibit a lifting that exactly matches this larger shift. Consider the distributions over the \CBAIspace arms
\begin{align*}
F_1 :&= \half \Ber(p_1 + 2\eps) + \half \UnifI, \quad\text{and}
\\ F_i :&= \half \Ber(p_i - 2\eps) + \half \UnifI \;\;\;\; \forall i \in \{2, \dots, k\}.
\end{align*}
We now present the malicious \CBAIspace adversarial strategy. For the optimal arm $1$, define the joint distribution $J_1$ over $(Y_1, Z_1, D_1)$ where $Y_1 \sim F_1$, $Z_1 \sim \delta_0$, and the conditional distribution of $D_1$ given $Y_1$ is $\Ber(\eps(\tfrac{p_1}{2} + \eps)^{-1} \mathds{1}\{Y_1 = 1\})$. Similarly, for each suboptimal arm $i \in \{2, \dots, k\}$, define the joint distribution $J_i$ over $(Y_i, Z_i, D_i)$ to be $Y_i \sim F_i$ and $Z_i \sim \delta_1$, and the conditional distribution of $D_i$ given $Y_i$ is $\Ber(\eps(\tfrac{1- p_1}{2} + \eps)^{-1} \mathds{1}\{Y_i = 0\})$. It is simple to see that for each arm $i \in [k]$, the marginals are correct under each $J_i$. Indeed, a simple conditioning calculation yields
\begin{align*}
\Prob (D_1 = 1)
&=
\Prob(D_1 = 1 | Y_1 = 1) \Prob(Y_1 = 1) + \Prob(D_1 = 1 | Y_1 \neq 1) \Prob(Y_1 \neq 1)
\\ &= \eps(\tfrac{p_1}{2} + \eps)^{-1} \cdot \half(p_1 + 2\eps) + 0
\\ &= \eps.
\end{align*}
An identical argument shows that the marginal distribution of $D_i$ is also equal to $\Ber(\eps)$ for each suboptimal arm $i \in \{2, \dots, k\}$. Now, for each arm $i \in [k]$, denote by $C_i$ the corresponding contaminated distributions induced by $(1-D_i)Y_i + D_iZ_i$ where $(Y_i,Z_i,D_i) \sim J_i$. It is not hard to see that
\begin{align*}
\tilde{F}_1 := \SBer(p_1) &= C_1,\quad\text{and} \\
\tilde{F}_i := \SBer(p_i) &= C_i, \;\;\;\; \forall i \in \{2, \dots, k\}
\end{align*}
In other words, \textit{samples generated from $\SBer(p_i)$ are equal in distribution to the samples generated by the malicious \CBAIspace  adversary's distribution $C_i$.}
\par A simple calculation shows that for the optimal arm, $(F_1)^{-1}(\half) = p_1 + 2\eps$ and $(\tilde{F}_1)^{-1}(\half) = (F_1)^{-1}(\half - \eps) = p_1$. Note further that $F_1$ has cdf satisfying $F_1(s) = \half(1 - p_1 - 2\eps) + \half s$ for all $s \in [0, 1)$, and that $F_1(1) = 1$. Thus, for any $x_1, x_2 \in [Q_{L,F_1}(\half - \tmax), Q_{R,F_1}(\half + \tmax)]$,  we have that $|F(x_1) - F(x_2)| = \frac{1}{2}|x_1 - x_2|$ since $[\half \plusminus \tmax] \subseteq [0.4, 0.6]$ is contained within the interval $[\half(p_1 + 2\eps), 1 - \half(p_1 + 2\eps) ] \supseteq [\half(\tfrac{2}{3} + 2 \cdot \tfrac{1}{15}), 1 - \half(\tfrac{2}{3} + 2 \cdot \tfrac{1}{15})] = [0.4, 0.6]$. By definition, this implies that $F_1 \in \mathcal{F}_{\tmax, B_1}$ with $B_1 m_2(F_1) = 2$. (Note that $B_1$ is a finite constant since $p_1 \in (\tfrac{1}{3}, \tfrac{2}{3}) \subset (0, 1)$ implies $m_2(F_i) > 0$.) A completely identical calculation shows that each suboptimal arm $i \in \{2, \dots, k\}$ satisfies $(F_i)^{-1}(\half) = p_i - 2\eps$, $(\tilde{F}_i)^{-1}(\half) = (F_i)^{-1}(\half + \eps) = p_i$, and $F_i \in \mathcal{F}_{\tmax, B_i}$ with $B_i m_2(F_i) = 2$. Therefore, we conclude that for each $i \inf [k]$, the difference in medians between the distributions $F_i$ and $\tilde{F}_i$ is equal to
\begin{align*}
|m_1(F_i) - m_1(\tilde{F}_i)| &= 2\eps = B_im_2(F_i)\eps = U_{\eps, B_i, m_2(F_i)}^{\textsc{(malicious)}},
\end{align*}
which is exactly equal to the largest possible uncertainty.An identical reduction argument as in the oblivious and prescient adversary finishes the proof.

\end{proof}